\documentclass{amsart}
\usepackage[margin=3.45cm]{geometry}

\usepackage{amssymb}
\usepackage{graphicx} 
\usepackage{mathrsfs}
\usepackage{enumerate}
\usepackage{xspace}
\usepackage{color}
\usepackage{enumerate}
\usepackage{enumitem}
\usepackage{amsthm}
\usepackage{dsfont}
\usepackage{bbm}
\usepackage[colorlinks=true, linkcolor = blue, citecolor = blue]{hyperref}
\usepackage[numbers]{natbib}
\usepackage[backgroundcolor=white,bordercolor=red]{todonotes}
\DeclareMathAlphabet{\mathpzc}{OT1}{pzc}{m}{it}
\usepackage[nameinlink]{cleveref}
\numberwithin{equation}{section}
\begin{document}

\theoremstyle{plain}

\newtheorem{theorem}{Theorem}[section]
\newtheorem{lemma}[theorem]{Lemma}
\newtheorem{example}[theorem]{Example}
\newtheorem{proposition}[theorem]{Proposition}
\newtheorem{corollary}[theorem]{Corollary}
\newtheorem{definition}[theorem]{Definition}
\newtheorem{Ass}[theorem]{Assumption}
\newtheorem{condition}[theorem]{Condition}
\theoremstyle{definition}
\newtheorem{remark}[theorem]{Remark}
\newtheorem{SA}[theorem]{Standing Assumption}
\newtheorem*{discussion}{Discussion}
\newtheorem{remarks}[theorem]{Remark}
\newtheorem*{notation}{Remark on Notation}
\newtheorem{application}[theorem]{Application}

\newcommand{\of}{[\hspace{-0.06cm}[}
\newcommand{\gs}{]\hspace{-0.06cm}]}

\newcommand\llambda{{\mathchoice
		{\lambda\mkern-4.5mu{\raisebox{.4ex}{\scriptsize$\backslash$}}}
		{\lambda\mkern-4.83mu{\raisebox{.4ex}{\scriptsize$\backslash$}}}
		{\lambda\mkern-4.5mu{\raisebox{.2ex}{\footnotesize$\scriptscriptstyle\backslash$}}}
		{\lambda\mkern-5.0mu{\raisebox{.2ex}{\tiny$\scriptscriptstyle\backslash$}}}}}

\newcommand{\1}{\mathds{1}}

\newcommand{\F}{\mathbf{F}}
\newcommand{\G}{\mathbf{G}}

\newcommand{\B}{\mathbf{B}}

\newcommand{\M}{\mathcal{M}}

\newcommand{\la}{\langle}
\newcommand{\ra}{\rangle}

\newcommand{\lle}{\langle\hspace{-0.085cm}\langle}
\newcommand{\rre}{\rangle\hspace{-0.085cm}\rangle}
\newcommand{\blle}{\Big\langle\hspace{-0.155cm}\Big\langle}
\newcommand{\brre}{\Big\rangle\hspace{-0.155cm}\Big\rangle}

\newcommand{\X}{\mathsf{X}}

\newcommand{\tr}{\operatorname{tr}}
\newcommand{\N}{{\mathbb{N}}}
\newcommand{\cadlag}{c\`adl\`ag }
\newcommand{\on}{\operatorname}
\newcommand{\oP}{\overline{P}}
\newcommand{\oO}{\mathcal{O}}
\newcommand{\D}{D(\mathbb{R}_+; \mathbb{R})}
\newcommand{\bx}{\mathsf{x}}
\newcommand{\z}{\mathfrak{z}}
\newcommand{\bb}{\hat{b}}
\newcommand{\bs}{\hat{\sigma}}
\newcommand{\bv}{\hat{v}}
\renewcommand{\v}{\mathfrak{m}}
\newcommand{\ob}{\widehat{b}}
\newcommand{\os}{\widehat{\sigma}}
\renewcommand{\j}{\varkappa}
\newcommand{\scl}{\ell}
\newcommand{\Y}{\mathscr{Y}}
\newcommand{\T}{\mathcal{T}}
\newcommand{\con}{\mathsf{c}}

\renewcommand{\epsilon}{\varepsilon}
\renewcommand{\rho}{\varrho}

\newcommand{\fPs}{\fP_{\textup{sem}}}
\newcommand{\fPas}{\mathfrak{S}^{\textup{ac}}_{\textup{sem}}}
\newcommand{\rrarrow}{\twoheadrightarrow}
\newcommand{\cA}{\mathcal{A}}
\newcommand{\ocA}{\mathcal{U}}
\newcommand{\cR}{\mathcal{R}}
\newcommand{\cK}{\mathcal{K}}
\newcommand{\cQ}{\mathcal{Q}}
\newcommand{\cF}{\mathcal{F}}
\newcommand{\cE}{\mathcal{E}}
\newcommand{\cC}{\mathcal{C}}
\newcommand{\cD}{\mathcal{D}}
\newcommand{\bC}{\mathbb{C}}
\newcommand{\cH}{\mathcal{H}}
\newcommand{\bth}{\overset{\leftarrow}\theta}
\renewcommand{\th}{\theta}
\newcommand{\cG}{\mathcal{G}}
\newcommand{\fPasn}{\mathfrak{S}^{\textup{ac}, n}_{\textup{sem}}}
\newcommand{\CLM}{\mathfrak{M}^\textup{ac}_\textup{loc}}

\newcommand{\bR}{\mathbb{R}}
\newcommand{\nnabla}{\nabla}
\newcommand{\f}{\mathfrak{f}}
\newcommand{\g}{\mathfrak{g}}
\newcommand{\oconv}{\overline{\on{co}}\hspace{0.075cm}}
\renewcommand{\a}{\mathfrak{a}}
\renewcommand{\b}{\mathfrak{b}}
\renewcommand{\d}{\mathsf{d}}
\newcommand{\bS}{\mathbb{S}^d_+}
\newcommand{\p}{\dot{\partial}}
\newcommand{\dr}{r} 
\newcommand{\m}{\mathbb{M}}
\newcommand{\Q}{Q}
\newcommand{\n}{\overline{\nu}} 
\newcommand{\usc}{\textit{USC}}
\newcommand{\lsc}{\textit{LSC}}
\newcommand{\q}{\mathfrak{q}}
\renewcommand{\X}{\mathscr{X}}
\newcommand{\W}{\mathcal{P}}
\newcommand{\fP}{\mathcal{P}}
\newcommand{\w}{\mathsf{w}}
\newcommand{\oM}{\mathsf{M}}
\newcommand{\oZ}{\mathsf{Z}}
\newcommand{\oK}{\mathsf{K}}
\renewcommand{\o}{{q}}
\renewcommand{\Re}{\operatorname{Re}}
\newcommand{\cCk}{\mathsf{c}_k}
\newcommand{\C}{\mathsf{C}}

\renewcommand{\emptyset}{\varnothing}

\allowdisplaybreaks

\makeatletter
\@namedef{subjclassname@2020}{%
	\textup{2020} Mathematics Subject Classification}
\makeatother

 \title[Propagation of Chaos for controlled McKean--Vlasov SPDEs]{Set-Valued Propagation of Chaos for controlled \\path-dependent McKean--Vlasov SPDEs} 
\author[D. Criens]{David Criens}
\author[M. Ritter]{Moritz Ritter}
\address{Albert-Ludwigs University of Freiburg, Ernst-Zermelo-Str. 1, 79104 Freiburg, Germany}
\email{david.criens@stochastik.uni-freiburg.de}
\email{moritz.ritter@stochastik.uni-freiburg.de}



\thanks{We are grateful to the anonymous referee for many helpful comments and suggestions, which helped to improve the manuscript.}
\date{\today}

\maketitle

\begin{abstract}
We develop a limit theory for controlled path-dependent mean field stochastic partial differential equations (SPDEs) within the semigroup approach of Da Prato and Zabczyk. More precisely, we prove existence results for mean field limits and particle approximations, and we establish set-valued propagation of chaos in the sense that we show convergence of sets of empirical distributions to sets of mean field limits in the Hausdorff metric topology. Furthermore, we discuss consequences of our results to stochastic optimal control. As another application, we deduce a propagation of chaos result for Peng's \(G\)-Brownian motion with drift interaction. 
\end{abstract}


\section{Introduction}
The area of controlled McKean--Vlasov dynamics, also known as mean field control, has rapidly developed in the past years, see, e.g., the monograph \cite{car_della} and the references therein.
Recently, there is also increasing interest in infinite dimensional systems such as controlled path-dependent McKean--Vlasov stochastic partial differential equations (controlled mean field SPDEs) of type 
\begin{align} \label{eq: SPDE intro}
d X_t = A X_t \, dt + b (\xi_t, t, X, P^X_t) \, dt + \sigma (\xi_t, t, X, P^X_t) \, d W_t, 
\end{align}
where \(\xi\) is a control process and \(P^X_t\) denotes the law of the stopped process \(X_{\cdot \wedge t}\). For controlled mean field SPDEs of type \eqref{eq: SPDE intro} with additional dependence on the distribution of the controls, well-posedness of the state equation, the dynamic programming principle and a Bellman equation were recently proved in the paper \cite{CGKPR23}. An infinite dimensional mean field control framework that also allows for jumps has been studied in \cite{DOS}. 
We also refer to these papers for comments on related literature and applications.

	Mean field dynamics are usually motivated by particle approximations, cf., e.g., Sznitman's seminal monograph \cite{SnzPoC}. 
It is important to establish this motivation in a rigorous manner.
For finite dimensional controlled systems, a general limit theory was developed in the paper \cite{LakSIAM17} and extended in \cite{DPT22} to a setup with common noise.  An infinite dimensional result was recently proved in the paper \cite{C23c} within the variational SPDE framework initiated by Pardoux~\cite{par75} and Krylov--Rozovskii~\cite{krylov_rozovskii}. 

In this paper, we continue this line of research in terms of a limit theory for controlled mean field SPDEs within the semigroup approach of Da Prato and Zabczyk \cite{DaPratoEd1}. Our results provide a particle motivation for equations of type \eqref{eq: SPDE intro} with feedback controls, which is a setting in the spirit of the paper \cite{CGKPR23}, albeit using a different control formulation.  Here, a control process is said to be a {\em feedback control} if it only depends on the state process in a predictable, but possibly path-dependent, way. We emphasize that it needs {\em not} be Markovian. At this point, we already highlight that our setting relates naturally to the other important relaxed and weak control frameworks, showing that our results also hold in these formulations. Parts of our results also translate to the strong control setup, relating our work to the paper~\cite{CGKPR23}, see Remark~\ref{rem: main1}~(ii) for a discussion.

To explain our main results, consider a particle system \(X = (X^1, \dots, X^n)\) given by mild solutions to the SPDEs 
\begin{align*}
d X^k_t &= A X^k_t  \, dt + b (\f^k_t (X), t, X^k, \X_n (X_{\cdot \wedge t})) \, dt + \sigma (\f^k_t (X), t, X^k, \X_n (X_{\cdot \wedge t})) \, d W^k_t, 
\end{align*}
with i.i.d. (= independent, identically distributed) initial data, whose distribution we denote by~\(\nu^n\), 
where
\begin{align*}
\X_n (X) &= \frac{1}{n} \sum_{k = 1}^n \delta_{X^k}
\end{align*}
denotes the empirical distribution of the particles, \(\f = (\f^1, \dots, \f^n)\) are arbitrary feedback controls and \(W^1, \dots, W^n\) are independent cylindrical Brownian motions. 
Here, the linear operator \(A\) is the generator of a strongly continuous semigroup on the Hilbert space $H$, which is the state space of the particles.
Let \(\ocA^n (\nu^n)\) be the set of empirical distributions of such particle systems.
The associated set of mean field limits is denoted by \(\cA^0 (\nu^0)\), where \(\nu^0\) denotes the limit of \(\nu^n\) in suitable Wasserstein topology. It consists of all laws of mild solutions to so-called McKean--Vlasov (or distribution dependent) SPDEs of the type
\[
d X_t = A X_t \, dt + b (\f_t (X), t, X, P^X_t) \, dt + \sigma (\f_t (X), t, X, P^X_t) \, dW_t, \quad X_0 \sim \nu^0, 
\]
where \(\f\) is an arbitrary feedback control and \(W\) is a cylindrical Brownian motion. Finally, let \(\ocA^0 (\nu^0)\) be the set of all probability measures that are supported on the set \(\cA^0 (\nu^0)\), i.e., 
\[
\ocA^0 (\nu^0) = \Big\{ P \colon P (\cA^0 (\nu^0)) = 1 \Big\}.
\]
We notice that \(\ocA^n (\nu^n)\) and $\ocA^0 (\nu^0)$ consist of probability measures on a set of probability measures, which appears to be natural due to the interest in laws of empirical distributions.

This setting can equivalently be framed in the context of stochastic processes under parameter uncertainty, also called nonlinear stochastic processes, as studied, e.g., in the recent papers \cite{C25_SIFIN, C23b,CN22a, CN22b, CN25, neufeld2017nonlinear, NVH, peng2007g}. This connection provides a comprehensive interpretation of our framework, formulating Knightian uncertainty within the notion of stochastic control.

Our contribution is twofold and investigates the relation of the sets \(\ocA^n (\nu^n)\) and \(\ocA^0 (\nu^0)\) from an analytic and a stochastic optimal control perspective. 

For the analytic part, we show that \(\ocA^n (\nu^n)\) and \(\ocA^0 (\nu^0)\) are nonempty and compact in a suitable Wasserstein space and that \(\ocA^n (\nu^n)\) converges to \(\ocA^0 (\nu^0)\) in the Hausdorff metric topology. This result can be interpreted probabilistically as {\em set-valued propagation of chaos}. Indeed, when \(\ocA^n (\nu^n)\) and \(\ocA^0 (\nu^0)\) are singletons, we recover the classical formulation of propagation of chaos. 
To the best of our knowledge, set-valued propagation of chaos was first studied in the recent paper \cite{C23c} for a variational controlled SPDE framework. 
The concept of set-valued propagation of chaos can also be put in the context of model uncertainty. It shows that the sets of feasible interacting stochastic models converge to their McKean--Vlasov counterparts in a meaningful topology. In this regard, the result provides a natural extension of the classical mean field theory to a setting with uncertainty. 
The non-emptiness of \(\ocA^0 (\nu^0)\) provides an existence result for controlled McKean--Vlasov SPDEs in a semigroup framework. In particular, it covers some uncontrolled cases that were studied in \cite{bhatt1998interacting,C23}.
Our proof for \(\ocA^0 (\nu^0) \not = \emptyset\) is based on a particle approximation, not relying on Lipschitz assumptions.

As a second main contribution, we investigate approximation properties of optimal control problems. Namely, for a continuous input function \(\psi\) of suitable growth, we prove that the value function 
\(\nu \mapsto \sup_{Q\,\in\, \ocA^n (\nu)} E^Q [ \psi  ]\) related to \(\ocA^n (\nu)\) converges uniformly on compacts (in their initial distributions \(\nu\)) to the value function \(\nu \mapsto \sup_{Q\, \in\, \ocA^0 (\nu)} E^Q [  \psi  ]\) related to \(\ocA^0 (\nu)\). We also derive versions of this statement for upper and lower semicontinuous input functions \(\psi\) of suitable growth. 
These results allow us to deduce limit theorems in the spirit of \cite{LakSIAM17}. Namely, we show that accumulation points of sequences of \(n\)-state nearly optimal controls maximize the mean field value function, and that any optimal mean field control can be approximated by a sequence of \(n\)-state nearly optimal controls. 
Furthermore, they can be translated into the language of model uncertainty. 
To illustrate this point of view, we deduce a type of propagation of chaos for \(G\)-Brownian motion with drift interaction. 

We now comment on related literature. As mentioned above, mean field SPDEs within the semigroup approach have been investigated in the recent paper \cite{CGKPR23}. A particle motivation for such a framework appears to be missing in the literature.  The objective of the present paper is to address this gap. For a comparison of the assumptions used, we refer to Remark \ref{rem: main1}~(iii) below. 

Our work is heavily inspired by the papers \cite{LakSIAM17} and \cite{C23c}. 
We highlight that the SPDE framework used in this paper is technically different from both, the finite dimensional setting that was investigated in \cite{LakSIAM17} and the variational framework studied in \cite{C23c}.
From a modeling point of view, the references \cite{C23c,LakSIAM17} work within a relaxed control framework (as used, e.g., in \cite{nicole1987compactification,EKNJ88, ElKa15}) and provide limit theorems for the joint empirical distributions of the particles and their controls, while we work with feedback controls and the empirical distributions of the particles. We prove that, under certain convexity assumptions, our setting can be translated to a relaxed framework. In particular, this shows that our main convergence results also hold for the relaxed as well as the weak control frameworks; see \cite{ElKa15}.
Working with feedback controls comes with some pleasant features. 
For example, it allows us to impose assumptions directly on the volatility coefficient~\(\sigma\), circumventing a type of disintegration procedure that was used in \cite{C23c,nicole1987compactification}. Furthermore, we mention again that our model allows for a novel interpretation in the context of model uncertainty and that some of our results also propagate to the strong control framework, building a connection to the recent paper~\cite{CGKPR23}, see Remark~\ref{rem: main1}~(ii).

Let us also comment on some technical aspects of our work. The semigroup framework differs from its variational counterpart in many points. 
For instance, in \cite{C23c} the state space for the paths of the particles is the intersection of a classical path space of continuous functions with an \(L^p\) space and it involves two Banach spaces. Here, we work only with one Hilbert space and the path space of continuous functions. Further, in the paper~\cite{C23c} certain uniform moment bounds are incorporated into the definition of the model. This is not necessary in our setting, as suitable estimates can be proved under linear growth conditions on the coefficients that appear natural in our setting. 
Such differences also influence the structure of the results and proofs. For example, our moment estimates enable us to prove compactness of the set \(\cA^0\), which then transfers directly to \(\ocA^0\). The setting from \cite{C23c} gave no access to compactness of \(\cA^0\).
Similar to \cite{C23c,LakSIAM17}, parts of our proofs rely on compactness and martingale problem methods that were developed in~\cite{nicole1987compactification} to study the regularity of value functions in a finite dimensional Markovian relaxed control setting.
In order to apply such methods, we relate our setting to a relaxed control framework. The proof for this connection relies on convexity arguments and Filippov's implicit function theorem.
Further, we adapt some tightness and martingale problem techniques from the papers~\cite{bhatt1998interacting,C23,gatarekgoldys} to our setup with controls.

This paper is structured as follows. Our framework and the main results are explained in Section~\ref{sec: main1}. 
The application of our main result to \(G\)-Brownian motion with drift interaction is presented in Section~\ref{app: G BM}.
The proofs are given in Section~\ref{sec: pf}. Furthermore, we added an appendix that provides a general existence result for classical SPDEs without controls. 

	\begin{notation}
	In this paper, \(C\) denotes a generic positive constant that might change from line to line. In case the constant depends on important quantities, this is mentioned specifically. 
\end{notation}

\section{Propagation of Chaos for controlled SPDEs} \label{sec: main1}
Fix a compact metrizable space \(F\), which is considered to be the action space for the control processes.
Let \(H\) be a separable Hilbert space (endowed with the norm topology), take a finite time horizon \(T > 0\) and let \(\Omega\) be the space of all continuous functions from \([0, T]\) into \(H\) endowed with the uniform topology. The coordinate map on \(\Omega\) is denoted by \(X = (X_t)_{t \in [0, T]}\). We define \(\cF := \sigma (X_t, t \in [0, T])\), which is well-known to be the Borel \(\sigma\)-field on \(\Omega\), and the corresponding filtration \(\mathbf{F} = (\cF_t)_{t \in [0, T]}\) with \(\cF_t := \sigma (X_s, s \in [0, t])\).
Take another separable Hilbert space~\(U\), which we use as state space for the randomness that drives our systems. The space of bounded linear operators from \(U\) into \(H\) is denoted by \(L (U, H)\) and the operator and Hilbert--Schmidt norm is denoted by \(\|\cdot\|_{L (U, H)}\) and  \(\|\cdot\|_{L_2 (U, H)}\), respectively. Further, in case \(U = H\) we suppress the second argument in our notation, i.e., for example we write  \(L (H)\) instead of \(L (H, H)\).
For any Polish space \(E\), let \(\fP(E) \equiv \W^0 (E)\) be the space of Borel probability measures on \(E\) and endow it with the weak topology, i.e., the topology of convergence in distribution. For \(t \in [0, T]\) and \(\omega \in \Omega\), we set 
\[
\|\omega\|_t := \sup_{s \in [0, t]} \|\omega (s)\|_H,
\]
and, for \(p \geq 1\), we define the \(p\)-Wasserstein space
\[
\W^p (\Omega) := \Big\{ \mu \in \fP (\Omega) \colon \|\mu\|_p := \Big( \int \|\omega\|_T^p \, \mu (d\omega) \Big)^{1/p} < \infty \Big\}.
\]
We endow \(\W^p (\Omega)\) with the \(p\)-Wasserstein topology that is generated by the \(p\)-Wasserstein metric~\(\w_p\). We define \(\W^p (H)\) in the same way with \(\|\, \cdot \,\|_T\) replaced by the norm \(\| \, \cdot \, \|_H\) of the Hilbert space \(H\).

Throughout this paper, we fix four constants \(\alpha, p, \o\) and $\rho$ such that 
\begin{align} \label{eq: constants}
\alpha \in \left( 0, \tfrac{1}{2}\right), \quad p\in \left(\tfrac{1}{\alpha},\infty\right), \quad \o \in \{0\}\cup \left[1,p\right), \text{ and }\rho\in \big(0,1-\tfrac{2}{p}\big).
\end{align}
Furthermore, we fix a Borel function \(\j \colon [0, T] \to [0, \infty]\) such that 
		\begin{align} \label{eq: DZ cond}
	\int_0^T \Big[\frac{\j (s)}{s^{\alpha}} \Big]^2 \, ds < \infty.
\end{align}
Let 
\begin{align*}
	&b \colon F \times [0, T] \times \Omega \times \W^\o (\Omega) \to H, 
	\\
	&\sigma \colon F \times [0, T] \times \Omega \times \W^\o (\Omega) \to L (U, H)
	\end{align*} 
	be Borel measurable functions. Furthermore, we presume that \(b\) and \(\sigma\) are predictable in the sense that, for all \((f, t, \omega, \mu) \in F \times [0, T] \times \Omega \times \W^\o(\Omega)\), 
\(
b (f, t, \omega, \mu)\) and \(\sigma (f, t, \omega, \mu)\)
depend on \(\omega\) only through \((\omega (s))_{s < t}\).
Let \(A \colon D(A) \subset H \to H\) be the generator of a strongly continuous semigroup \((S_t)_{t \geq 0}\) on \(H\). 

\smallskip 
We proceed with the formulation of the conditions needed for our main result.

\begin{condition} \label{cond: main1}
	\quad
	\begin{enumerate}
		\item[\textup{(i)}] The functions \(b\) and \(\sigma\) are continuous on \(F \times [0, T] \times \Omega \times \W^\o (\Omega)\).
		\item[\textup{(ii)}] There exists a constant \(C > 0\) such that 
		\begin{align}
			\|b (f, t, \omega, \mu) \|_H + \|\sigma (f, t, \omega, \mu)\|_{L (U, H)} & \leq C \Big[ 1 + \|\omega\|_t + \|\mu\|_p \Big],  \label{eq: LG1}
			\\
			\|S_s \sigma (f, t, \omega, \mu) \|_{L_2 (U, H)} &\leq \j (s) \Big[ 1 + \|\omega\|_t + \|\mu\|_p \Big],  \label{eq: LG2}
		\end{align}
		for all \(f \in F, s, t \in [0, T], \omega \in \Omega\) and \(\mu \in \W^p (\Omega)\).
		\item[\textup{(iii)}] For every \((t, \omega, \mu) \in [0, T] \times \Omega \times \W^p (\Omega)\), the set 
		\[
		\big\{ (b (f, t, \omega, \mu), \sigma \sigma^* (f, t, \omega, \mu)) \colon f \in F \big\} \subset H \times L (H)
		\]
		is convex. Here, \(\sigma^*\) denotes the adjoint of \(\sigma\).
			\end{enumerate}
	\end{condition}

 \begin{condition}\label{cond: compact assumption}
     The operator \(A\) generates a compact semigroup, i.e., for every \(t > 0\), the operator \(S_t\) is compact.
 \end{condition}

Let us recall some concepts from functional analysis. We start with the definition of a Riesz basis, see \cite[Definition 7.9]{heil2010basis}.
\begin{definition}
A sequence $(e_k)_{k = 1}^\infty\subset H$ is called a \emph{Riesz basis} if it is equivalent to an orthonormal basis in $H$, i.e., there is a topological isomorphism $\T$ and an orthonormal basis $(b_k)_{k = 1}^\infty$ in $H$ such that $e_k= \T(b_k)$ for all $k\in \mathbb N$.
\end{definition}
\begin{remark} \label{rem: biorthogonal seq}
If $(e_k)_{k = 1}^\infty$ is a Riesz basis of $H$, then there are constants $C,c<\infty$ such that for all $f\in H$:
\begin{align}\label{eq: Bessel estimation}
   c\, \| f\|_H^2\leq \sum_{k = 1}^\infty|\langle f,e_k\rangle_H |^2\leq C\, \| f\|_H^2.
\end{align}
These constants are optimally defined via the operator norm of the topological isomorphism, i.e., $c=\|\T^{-1}\|^{-2}_{L(H)}$ and   $C=\|\T\|_{L(H)}^2$. 
For each Riesz basis $(e_k)_{k = 1}^\infty$ there exists an equivalent inner product $(\cdot,\cdot)$ on $H$ such that $(e_k)_{k = 1}^\infty$ is an orthonormal basis for $H$ with respect to $(\cdot,\cdot)$.
For the proofs see \cite[Lemma 7.12, Theorem 7.13]{heil2010basis}. 
\end{remark}

\begin{condition}\label{cond: iii alternative}
There is a Riesz basis $(e_k)_{k=1}^\infty \subset H$ with the following properties:
\begin{enumerate}
    \item[\textup{(i)}] There exists a sequence $(\lambda_k)_{k=1}^\infty \subset \mathbb{R}$ such that $\lambda_k>0$ and 
    \begin{align}\label{eq: eigenvector S*}
        S_t^{*}e_k=e^{-\lambda_k t}e_k\quad \text{for all $k\in\mathbb N$}.
    \end{align}
    \item[\textup{(ii)}] There exists a sequence $(\cCk)_{k=1}^\infty\subset \mathbb R_{+}$ such that
    \begin{align}\label{eq: summation_condition}
        \sum\limits_{k=1}^\infty  \cCk^2 \lambda_k^{-\rho} <\infty,
    \end{align}
    and
    \begin{align}\label{eq: estimate_b_sigma_in_summation_condition}
        |\langle b(f,t,\omega,\mu),e_k\rangle_H |^2 
        + \|\sigma^*(f,t,\omega,\mu)e_k\|^2_{U}&\leq \cCk^2\, \Big[1+\|\omega\|_t^2+\|\mu\|^2_{p}\Big]
    \end{align}
    for all $(f, t, \omega, \mu, k)\in F \times [0, T] \times \Omega \times \W^p (\Omega) \times \mathbb N$.
\end{enumerate}
\end{condition}
\begin{remark}
\begin{enumerate}
    \item[\textup{(i)}] 
    By \cite[Corollary~10.6, p. 41]{pazy}, the adjoint semigroup \((S^*_t)_{t \geq 0}\) is strongly continuous with generator \(A^*\). Moreover, if \(e_k\) is an eigenvector of \(- A^*\) for the eigenvalue \(\lambda_k\), then, by the exponential formula \cite[Theorem~8.3, p. 33]{pazy}, it holds that
    \begin{align*}
        S^*_t e_k &= \lim_{n\to\infty} \left(\operatorname{Id} -\,  tA^*/n\right)^{-n} e_k \\
        &= e^{-t\lambda_k}\lim_{n\to\infty} \left(\operatorname{Id}  -\,  tA^*/n\right)^{-n} \left(1+t\lambda_k/n \right)^n e_k \\
        &= e^{-t\lambda_k}\lim_{n\to\infty} \left(\operatorname{Id} -\,  tA^*/n\right)^{-n} \left(\operatorname{Id}  -\, tA^*/n\right)^n e_k \\  
        &= e^{-t\lambda_k}e_k.
    \end{align*}
    \item[\textup{(ii)}] Suppose that \(- A\) is a positive self-adjoint operator with purely discrete spectrum, as considered, for instance, in \cite{bhatt1998interacting}.  By \cite[Propositions 5.12, 5.13]{schmued}, there exists a sequence \((\lambda_k)_{k = 1}^\infty \subset \mathbb{R}_+\) such that \(\lim_{n \to \infty} \lambda_n = \infty\) and an orthonormal basis \((e_k)_{k = 1}^\infty \subset H\) such that 
    \begin{align*}
        A e_k = - \lambda_k e_k, \quad k \in \mathbb{N}.
    \end{align*}
    By virtue of \cite[Propositions~6.13, 6.14]{schmued}, thanks to the self-adjointness of \(A\), the operator \(A\) generates a contraction semigroup \((S_t)_{t \geq 0}\) of self-adjoint operators, and it holds that
    \begin{align*}
        S^*_t e_k = S_t e_k = e^{- \lambda_kt} e_k, \quad k \in \mathbb{N}.
    \end{align*}
    \item[\textup{(iii)}]
			A typical choice for \(\j\) from \eqref{eq: DZ cond} is the function \(t \mapsto \|S_t\|_{L_2 (H)}\). In this case the integrability condition \eqref{eq: DZ cond} translates to the classical Da Prato--Zabczyk condition (cf. Section~7.1.1~in~\cite{DaPratoEd1}) that is given by
			\begin{align} \label{eq: real DZ cond}
				\int_0^T \frac{\|S_s\|^2_{L_2 (H)}}{s^{2 \alpha}} \, ds < \infty.
			\end{align}
			In particular, \eqref{eq: real DZ cond} entails that \(S_t\) is compact for every \(t > 0\), i.e., it implies Condition~\ref{cond: compact assumption}.
			Further, in this situation, \eqref{eq: LG2} is implied by \eqref{eq: LG1}.
   
\smallskip
A concrete example where \eqref{eq: real DZ cond} holds is \(H = L^2 (\mathscr{O})\), for a bounded region \(\mathscr{O} \subset \mathbb{R}^d\) with smooth boundary, and \(A\) being strongly elliptic of order \(2m > d\), see \cite[Example~3]{gatarekgoldys}. This includes for instance the Laplacian in case \(d = 1\).

\item[\textup{(iv)}]
     Under Condition~\ref{cond: iii alternative}, the inequality \eqref{eq: LG2} holds for the choice
    \begin{align} \label{eq: choice kappa}
        \j (t) \equiv C \, \sqrt{ \sum_{k = 1}^\infty e^{- 2 \lambda_k t} \cCk^2 }, \quad t \in [0, T],
    \end{align}
    where \(C \geq 1\) is a large enough constant. 
			Indeed, by Lemma \ref{lem: hilbert schmidt estimate} below, this follows from the estimate
   \begin{align*}
       \| S_t \sigma (f, t, \omega, \mu) \|^2_{L_2 (H)} &\leq C\sum_{k = 1}^\infty \| \sigma^* (f, t, \omega, \mu) S^*_t e_k \|^2_U 
       \\&= C\sum_{k = 1}^\infty e^{- 2 \lambda_k t} \| \sigma^* (f, t, \omega, \mu) e_k \|^2_U 
       \\&\leq C \sum_{k = 1}^\infty e^{- 2 \lambda_k t} \, \cCk^2\, \Big[ 1+\|\omega\|_t^2+\|\mu\|^2_{p} \Big].
   \end{align*}
   Furthermore, with \(\kappa\) as in \eqref{eq: choice kappa}, \eqref{eq: DZ cond} holds for \(\alpha = (1 - \rho) / 2\), as
   \begin{align*}
       \int_0^T \Big[ \frac{\j (s)}{s^{ (1 - \rho) / 2}} \Big]^2 \, ds &= C \, \sum_{k = 1}^\infty \int_0^T \frac{e^{- 2 \lambda_k s} \cCk^2}{s^{1 - \rho}} \, ds
       \\&= C \, \sum_{k = 1}^\infty \int_0^{\lambda_k T} \frac{e^{- 2 z}}{z^{1 - \rho}} \, dz \, \frac{\cCk^2}{\lambda_k^\rho} 
       \\&\leq C \, \int_0^\infty \frac{e^{- 2z}}{z^{1 - \rho}} \, dz \, \sum_{k = 1}^\infty \cCk^2 \lambda_k^{- \rho} < \infty.
   \end{align*}
   Notice that this choice of \(\alpha\) is in line with \eqref{eq: constants}, since, for \(\alpha = (1 - \rho) / 2\), \(p > 1 / \alpha\) is equivalent to \(\rho < 1 - 2 / p\), and \(\alpha < 1 / 2\) holds if and only if \(\rho > 0\).
   \item[\textup{(v)}] Let \(L\) be a closed, densely defined linear operator with simple eigenvalues \((\lambda_n)_{n = 1}^\infty\) and corresponding eigenvectors \((e_n)_{n = 1}^\infty\) that are assumed to form a Riesz basis. According to \cite[Exercise~3.20, p. 145]{CZ20}, \(L\) has a compact resolvent if and only if \(\lim_{n \to \infty}  1 / \lambda_n = 0\). 
   
   \smallskip 
   By virtue of \cite[Theorem~3.3, p. 48]{pazy}, compactness of the resolvent of \(A\) is a necessary condition for the compactness of the semigroup \((S_t)_{t \geq 0}\) and therefore, for Condition~\ref{cond: compact assumption}.
   Consequently, in case \(A\) has eigenvectors \((e_n)_{n = 1}^\infty\), corresponding to simple eigenvalues \((\lambda_n)_{n = 1}^\infty\), that form a Riesz basis, \(\lim_{n \to \infty} 1 / \lambda_n = 0\) is necessary for Condition~\ref{cond: compact assumption}. This distinguishes Condition~\ref{cond: compact assumption} from Condition~\ref{cond: iii alternative}, where such an assumption is not needed.
   
   \smallskip
   Let us provide an explicit example. Take \(H = \ell^2\) and let \((q (n))_{n = 1}^\infty \subset \bR\) be such that \(\sup_{n \in \mathbb{N}} q (n) < \infty\). We emphasize that the sequence \((q(n))_{n = 1}^\infty\) is not assumed to be bounded from below. The so-called {\em multiplicative semigroup} is given by 
   \[
   S_t x := e^{ t q} x=(e^{tq(n)}x(n))_{n=1}^\infty, \quad t \geq 0, \ x \in \ell^2.
   \]
   It is well-known that \((S_t)_{t \geq 0}\) is a strongly continuous semigroup with generator 
   \[
   A x = q  x= (q(n)x(n))_{n=1}^\infty, \quad x \in D(A) := \big\{ x \in \ell^2 \colon q  x \in \ell^2 \big\},
   \]
   cf. \cite[Section~II.2.b]{EN}. 
   Evidently, \((S_t)_{t \geq 0}\) and \(A\) are self-adjoint.
   For \(n \in \mathbb{N}\), define \(\lambda_n := - q (n)\) and \(e_n (k) := \1_{\{k = n\}}\) for \(k \in \mathbb{N}\). Then, \((\lambda_n)_{n = 1}^\infty\) are eigenvalues of \(- A\) with corresponding eigenvectors \((e_n)_{n = 1}^\infty\). In particular, by the discussion above (or see the proposition on p. 122 in \cite{EN}), \((S_t)_{t \geq 0}\) is compact only when \(\lim_{n \to \infty} q (n) = - \infty\). Depending on the sequences \((\mathsf{c}_n)_{n = 1}^\infty\) and \((q (n))_{n = 1}^\infty\), it is possible that Condition~\ref{cond: iii alternative} holds although \(\lim_{n \to \infty} q (n) = - \infty\) is violated.
\end{enumerate}  
    \end{remark}
\begin{condition} \label{cond: main2}	
	There exists a constant \(C > 0\) such that 
		\begin{align*}
			\| b (f, t, \omega, \mu) - b (f, t, \alpha, \nu) \|_ H &\leq C \big( \|\omega - \alpha\|_t + \w_p (\mu, \nu)\big), \\
			\| S_s (\sigma (f, t, \omega, \mu) - \sigma (f, t, \alpha, \nu)) \|_ {L_2 (U, H)} &\leq \j (s) \big( \|\omega - \alpha\|_t + \w_p (\mu, \nu)\big),
		\end{align*}
		for all \(f \in F, s, t \in [0, T], \omega, \alpha \in \Omega\) and \(\mu, \nu \in \W^p (\Omega)\).
\end{condition}

For \(n \in \mathbb{N}\), define 
\begin{align} \label{eq: X_n def}
\X_n \colon \Omega^n \to \fP (\Omega), \quad \X_n (\omega^1, \dots, \omega^n) := \frac{1}{n} \sum_{k = 1}^n \delta_{\omega^k}.
\end{align}
The following definition introduces a set of interacting SPDEs with feedback controls. 
\begin{definition} 
For \(\nu \in \W(H)\) and \(n \in \mathbb{N}\), let \(\cA^n (\nu)\) be the set of probability measures \(P \in \fP(\Omega^n)\) such that there exist \(\mathbf{F}^n\)-predictable processes \(\f^1, \dots, \f^n \colon [0, T] \times \Omega^n \to F\) and, possibly on a standard extension of the stochastic basis \((\Omega^n, \cF^n, \mathbf{F}^n, P)\), independent standard cylindrical Brownian motions \(W^1, \dots, W^n\) such that \(P\)-a.s., for all \(t \in [0, T]\) and \(k = 1, \dots, n\),
    \[
    X^k_t = S_t X^k_0 + \int_0^t S_{t - s} b (\f^k_s, s, X^k, \X_n (X_{\cdot \wedge s})) \, ds + \int_0^t S_{t - s} \sigma (\f^k_s, s, X^k, \X_n (X_{\cdot \wedge s})) \, d W^k_s, 
    \]
    and \(X^1_0, \dots, X^n_0\) are i.i.d.\footnote{= independent, identically distributed}  with distribution \(\nu\), where \(X = (X^1, \dots, X^n)\) denotes the coordinate process on \(\Omega^n\).
\end{definition}

Next, we also define the set of potential mean field control limits. 
\begin{definition} 
	\label{def: cA0}
For \(\nu \in \W^p (H)\), let \(\cA^0 (\nu)\) be the set of probability measures \(P \in \W^p(\Omega)\) such that there exist an \(\mathbf{F}\)-predictable process \(\f \colon [0, T] \times \Omega \to F\) and, possibly on a standard extension of the stochastic basis \((\Omega, \cF, \mathbf{F}, P)\), a standard cylindrical Brownian motion \(W\) such that \(P\)-a.s., for all \(t \in [0, T]\),
    \[
    X_t = S_t X_0 + \int_0^t S_{t - s} b (\f_s, s, X, P^X_s) \, ds + \int_0^t S_{t - s} \sigma (\f_s, s, X, P^X_s) \, d W_s,
    \]
    and \(X_0 \sim \nu\),\footnote{that is, \(X_0\) has distribution \(\nu\)} 
    where \(P^X_s := P \circ X^{-1}_{\cdot \wedge s}\).
\end{definition}

In the following, we investigate the connection of the sets
\begin{align*}
\ocA^n (\nu) &:= \Big\{ Q \in \fP (\fP (\Omega)) \colon Q = P \circ  \X_n^{-1} \text{ for some } P \in \cA^n (\nu) \Big\}, \\ 
\ocA^0 (\nu) &:= \Big\{ Q \in \fP (\fP(\Omega)) \colon Q ( \cA^0 (\nu)) = 1 \Big\}.
\end{align*}
At this point, anticipating Theorem~\ref{theo: main1}~(i) below, the set \(\cA^0 (\nu)\) is compact in \(\fP^\o (\Omega)\). Thus, \(\cA^0 (\nu) \in \mathcal{B} (\fP^\o (\Omega)) \subset \mathcal{B} (\fP (\Omega))\) (see \cite[Discussion after Proposition~5.7]{car_della}) and \(\ocA^0 (\nu)\) is well-defined. 
The following theorem is the main result of this paper. 
\begin{theorem} \label{theo: main1}
	We impose Condition \ref{cond: main1}. In addition, we assume that either Condition~\ref{cond: compact assumption} or Condition~\ref{cond: iii alternative} holds. 
 Take a sequence \((\nu^n)_{n = 0}^\infty \subset \W^p (H)\) such that \(\sup_{n \geq 0} \int \|z\|^p_H \, \nu^n (dz) < \infty\) and \(\nu^n \to \nu^0\) in~\(\W^q (H)\). Then, the following hold:
 
	\begin{enumerate}
		\item[\textup{(i)}] For every \(n \in \mathbb{N}\), the sets \(\cA^0 (\nu^0) \subset \W^\o (\Omega), \cA^n (\nu^n) \subset \W^\o (\Omega^n)\) and \(\ocA^0 (\nu^0), \ocA^n (\nu^n) \subset \W^\o (\W^\o (\Omega))\) are nonempty and compact. 
		\item[\textup{(ii)}] Every sequence \((Q^n)_{n = 1}^\infty\) with \(Q^n \in \ocA^n (\nu^n)\) is relatively compact in \(\W^\o (\W^\o (\Omega))\) and each of its \(\o\)-Wasserstein accumulation points is in \(\ocA^0 (\nu^0)\).
		\item[\textup{(iii)}] For every upper semicontinuous function \(\psi \colon \W^\o (\Omega) \to \bR\) such that 
		\begin{align} \label{eq: bound property}
			\exists \, C > 0 \colon \ \ |\psi (\mu)| \leq C (1 + \|\mu\|^\o_\o) \ \ \forall \, \mu \in \W^\o (\Omega),
			\end{align}
		it holds that
		\[
		\limsup_{n \to \infty} \sup_{Q \in \ocA^n (\nu^n)} E^Q \big[ \psi \big] \leq \sup_{Q \in \ocA^0 (\nu^0)} E^Q \big[ \psi \big].
		\]
	\end{enumerate}
	In addition to the assumptions above, suppose that Condition \ref{cond: main2} holds, and that \(\nu^n \to \nu^0\) in~\(\W^p (H)\). 	
	\begin{enumerate}
		\item[\textup{(iv)}] 
		For every \(Q^0 \in \ocA^0 (\nu^0)\), there exists a sequence \((Q^n)_{n = 1}^\infty\) with \(Q^n \in \ocA^n (\nu^n)\) and \(Q^n \to Q^0\) in \(\W^\o (\W^\o (\Omega))\).
		\item[\textup{(v)}] For every lower semicontinuous function \(\psi \colon \W^\o (\Omega) \to \bR\) with the property \eqref{eq: bound property}, it holds that 
		\[
		 \sup_{Q \in \ocA^0 (\nu^0)} E^Q \big[ \psi \big] \leq \liminf_{n \to \infty} \sup_{Q \in \ocA^n (\nu^n)} E^Q \big[ \psi \big].
		\]
				\item[\textup{(vi)}] For every compact set \(K \subset \W^p (H)\) and every continuous function \(\psi \colon \W^\o (\Omega) \to \bR\) with the property \eqref{eq: bound property}, it holds that
				\begin{align} \label{eq: compact convergence}
				\sup_{\nu \in K } \Big| \sup_{Q\in \ocA^n (\nu)} E^Q \big[ \psi \big] - \sup_{Q\in \ocA^0(\nu)} E^Q \big[ \psi \big] \Big| \to 0, \quad n \to \infty,
				\end{align}
			and the map 
			\[
			\nu \mapsto \sup_{Q\in \ocA^0(\nu)} E^Q \big[ \psi \big]
			\]
			is continuous from \(\W^p (H)\) into \(\bR\).
		\item[\textup{(vii)}] Let 
		\[
		\mathsf{h} (A, B) := \max \Big\{ \sup_{a \in A} \widehat{\w}_q (a, B), \, \sup_{b \in B} \widehat{\w}_q (b, A) \Big\}, \quad A, B \subset \W^\o (\W^\o (\Omega)), 
		\]
		be the Hausdorff metric on the space of nonempty compact subsets of \((\W^\o(\W^\o (\Omega)), \widehat{\w}_\o)\).\footnote{Here, \( \widehat{\w}_\o\) denotes the \(q\)-Wasserstein metric on \(\W^\o (\W^\o(\Omega))\), and recall that the sets \(\ocA^n (\nu^n)\) and \(\ocA^0 (\nu^0)\) are nonempty and compact in \(\W^\o (\W^\o (\Omega))\) thanks to part (i) of the theorem.}
		Then, for every compact set \(K \subset \W^p (H)\), 
		\[
		\sup_{\nu \in K } \mathsf{h} (\ocA^n(\nu), \ocA^0(\nu)) \to 0, \quad n \to \infty.
		\]
  Furthermore, the map \(\nu \mapsto \ocA^0(\nu)\) is continuous from \(\W^p (H)\) into the space of nonempty compact subsets of \(\W^\o(\W^\o (\Omega))\) with the Hausdorff metric topology.
	\end{enumerate}
\end{theorem}

\begin{remarks} \label{rem: main1}
(i)
  The probabilistic main result from Theorem~\ref{theo: main1} is part (vii). It can be seen as {\em set-valued propagation of chaos}. Indeed, in case the sets \(\cA^n (\nu^n) = \{Q^n\}\) and \(\cA^0 (\nu^0) = \{Q^0\}\) are singletons, Theorem~\ref{theo: main1}~(vii) implies that 
		\[
		Q^n \circ \X_n^{-1} \to \delta_{Q^0}
		\]
		in \(\W^\o (\W^\o (\Omega))\), which is classical propagation of chaos. 

Parts (i) and (ii) from Theorem~\ref{theo: main1} provide an existence result and particle approximations for controlled mean field SPDEs. In this regard, they include some results from \cite{bhatt1998interacting,C23} on the uncontrolled situation. 

Meanwhile, parts (iii), (v), and (vi) establish connections with stochastic optimal control theory. They contribute insights into mean field control problems and their respective approximations.
We will continue the discussion in Corollary~\ref{coro: control} below.
  
	\smallskip
\noindent
(ii)
A version of Theorem~\ref{theo: main1} within the variational framework for SPDEs has recently been established in the paper \cite{C23c}. Besides the distinct mathematical framework, the approaches are different in the sense that here we deal with feedback controls, while more general relaxed controls are considered in \cite{C23c}, see also \cite{nicole1987compactification,EKNJ88, ElKa15,LakSIAM17}. 
		  As already mentioned in the introduction, the convexity assumptions from Condition~\ref{cond: main1}~(iii) allow us to translate our feedback setting into a relaxed framework, see Section~\ref{sec: rel con}, especially Lemma~\ref{lem: nonlinear SPDE rel RCR}, below. 
		  In particular, Lemma~\ref{lem: nonlinear SPDE rel RCR} implies that Theorem~\ref{theo: main1} also holds for the relaxed control framework, which allows more flexible randomized controls.

  In general, working with feedback controls resembles with the concept of {\em model ambiguity} or {\em Knightian uncertainty} as considered in the recent papers \cite{C25_SIFIN, C23b,CN22a, CN22b, CN25, neufeld2017nonlinear, NVH, peng2007g}, for example.
		These papers extend the notion of Peng's \(G\)-Brownian motion (see, e.g., \cite{peng2007g}) to more general classes of stochastic processes. The relation to our framework is made precise by \cite[Proposition~2.4]{C23b}. 
  In Section~\ref{app: G BM} below, we discuss this in detail and present how Theorem~\ref{theo: main1}~(vi) can be used to establish a propagation of chaos result for \(G\)-Brownian motions.			
  
 Alongside the feedback and relaxed frameworks, the weak and strong control frameworks are frequently used in the literature; see \cite[Section~4.4]{ElKa15} for a discussion of their relations in the finite dimensional setting. In an infinite dimensional controlled SPDE context, the strong framework has been employed in \cite{CGKPR23}. The weak control framework is intermediate between the feedback and relaxed frameworks in the sense that every feedback control is a weak control, and every weak control can be associated with a relaxed control. Consequently, by Lemma~\ref{lem: nonlinear SPDE rel RCR} below, Theorem~\ref{theo: main1} also holds for the weak and relaxed control frameworks. Note that, even under very strong regularity assumptions on the coefficients, there exist feedback controlled processes that do not admit strong solutions and therefore fall outside the strong framework; see Tsirel'son's example \cite[Section~V.18]{RW}. As a consequence, neither the complete Theorem~\ref{theo: main1} nor all of our methods extend immediately to the strong setting. However, certain implications of Theorem~\ref{theo: main1} can be carried over to the strong setting. To be more precise, under Conditions~\ref{cond: main1} and \ref{cond: main2}, one may show the following two facts:
 	\begin{enumerate}
 		\item[(a)] Every feedback controlled process may be approximated in law by a strongly controlled process. This also holds for mean field systems.
 		\item[(b)] For lower semicontinuous payoffs, the value functions coincide for the feedback and strong control frameworks. 
 	\end{enumerate}
The fact (b) shows that Theorem~\ref{theo: main1}~(v) and (vi) carry over to the strong control framework, while (a) entails that the approximation sequence in (iv) can be chosen to consist of strong controls.
Finally, let us sketch the ideas behind (a) and (b). 
First, (b) follows from (a) and the lower semicontinuity of the payoff, since the value function for the strong control framework is bounded from above by that for the relaxed framework, which coincides with the feedback value function thanks to Lemma~\ref{lem: nonlinear SPDE rel RCR}. To show fact (a) one may argue as in the proof of \cite[Theorem~2.4]{LakSIAM17}, using the approximation result \cite[Lemma~3.11]{CDL_16}, which entails that feedback controls can be approximated by strong controls in an appropriate sense. Then, using the Lipschitz assumptions from Condition~\ref{cond: main2}, the approximation propagates to the controlled processes by a Gronwall argument; see the proof of  \cite[Theorem~2.4]{LakSIAM17} for details in the finite dimensional case.

	\smallskip
\noindent
(iv)
			We now comment on our conditions and relate them to technical assumptions imposed in the recent paper \cite{CGKPR23}. 
		 On the level of the coefficients \(b\) and \(\sigma\), the Assumption~\((\on{A}_{A, b, \sigma})\) from \cite{CGKPR23} imposes Lipschitz and linear growth conditions that are comparable to Condition~\ref{cond: main1}~(ii) and Condition~\ref{cond: main2}. One important difference lies in the norm chosen for the coefficient~\(\sigma\). The paper \cite{CGKPR23} uses the Hilbert--Schmidt norm for \(\sigma\), while we only use it for the modified coefficient \(S \sigma\).  This relaxation covers, for example, stochastic Cauchy problems
		of the form
		\[
		d Y_t = A Y_t \, dt + b (Y_t) \, dt + \sigma \, d W_t
		\] 
		  with \(\sigma \equiv \on{Id}\) and standard cylindrical noise \(W\). In infinite dimensional situations, Hilbert--Schmidt assumptions on the volatility coefficient would exclude the choice \(\sigma \equiv \on{Id}\), which means that the noise needs to be colored. 
	
	In \cite{CGKPR23}, the linearity \(A\) is assumed to generate a pseudo-contraction semigroup. Here, we impose either Condition~\ref{cond: compact assumption} or Condition~\ref{cond: iii alternative}. 
	We need one of these conditions for a tightness argument to establish relative compactness of the sets \(\cA^n (\nu^n)\) and \(\cA^0 (\nu^0)\). Further, we use these assumptions to prove that these sets are nonempty, which we do {\em without} Lipschitz conditions.

	Finally, we remark that in \cite{CGKPR23} the action space \(F\) is only assumed to be a Borel space, while here we presume it to be compact and metrizable (which entails that it is Polish and in particular a Borel space). 
\end{remarks}

Next, we deduce observations related to \cite[Theorems 2.11, 2.12]{LakSIAM17}. 
The first part of the following corollary shows that all accumulation points of \(n\)-state optimal controls are mean field optimal, while the second part explains that every optimal mean field control can be approximated by \(n\)-state nearly optimal controls. 

A version of the following corollary within the variational framework for SPDEs can be found in \cite{C23c}. The proof requires no change and, for reader's convenience, we outline the main steps.
\begin{corollary} \label{coro: control}
	Suppose that the Conditions~\ref{cond: main1} and \ref{cond: main2} hold, and impose either Condition~\ref{cond: compact assumption} or Condition~\ref{cond: iii alternative}.
 Take a continuous function \(\psi \colon \W^\o (\Omega) \to \bR\) with the property \eqref{eq: bound property} and
  initial distributions \((\nu^n)_{n = 0}^\infty \subset \W^p (H)\) with \(\nu^n \to \nu^0\) in \(\W^p (H)\).
	\begin{enumerate}
		\item[\textup{(i)}] Let \((\varepsilon_n)_{n = 1}^\infty \subset \bR_+\) be a sequence such that \(\varepsilon_n \to 0\). For each \(n \in \mathbb{N}\), suppose that \(Q^n \in \mathcal{U}^n (\nu^n)\) is such that 
		\[
		\sup_{Q \in \ocA^n (\nu^n)} E^{Q} \big[ \psi \big] - \varepsilon_n \leq E^{Q^n} \big[ \psi \big].
		\]
		In other words, \(Q^n\) is a so-called \(n\)-state \(\varepsilon_n\)-optimal control.
		Then, the sequence \((Q^n)_{n = 1}^\infty\) is relatively compact in \(\fP^q (\fP^q(\Omega))\) and every accumulation point \(Q^0\) is in \(\mathcal{U}^0 (\nu^0)\) and optimal in the sense that 
		\begin{align} \label{eq: def opti}
			E^{Q^0} \big[ \psi \big] = \sup_{Q \in \mathcal{U}^0 (\nu^0)} E^Q \big[ \psi \big].
		\end{align}
		\item[\textup{(ii)}] 
		Take a measure \(Q^0 \in \mathcal{U}^0 (\nu^0)\) that is optimal (i.e., it satisfies \eqref{eq: def opti}). Then, there are sequences \((\varepsilon_n)_{n = 1}^\infty \subset \bR_+\) and \((Q^n)_{n = 1}^\infty \subset \fP^q (\fP^q(\Omega))\) such that \(\varepsilon_n \to 0\), each \(Q^n\) is an \(n\)-state \(\varepsilon_n\)-optimal control and \(Q^n \to Q^0\) in \(\fP^q (\fP^q(\Omega))\).
	\end{enumerate}
\end{corollary}
\begin{proof}
	(i). By Theorem \ref{theo: main1} (vi), we have 
	\[
	\sup_{Q \in \ocA^0 (\nu^0)} E^Q \big[ \psi \big] \leftarrow \sup_{Q \in \ocA^n (\nu^n)} E^{Q} \big[ \psi \big] - \varepsilon_n \leq E^{Q^n} \big[ \psi \big] \leq \sup_{Q \in \ocA^n (\nu^n)} E^Q \big[ \psi \big] \to \sup_{Q \in \ocA^0 (\nu^0)} E^Q \big[ \psi \big],
	\]
	which implies
	\[
	\lim_{n \to \infty} E^{Q^n} \big[ \psi \big] = \sup_{Q \in \ocA^0 (\nu^0)} E^Q \big[ \psi \big].
	\]
	By part (ii) of Theorem \ref{theo: main1}, \((Q^n)_{n = 1}^\infty\) is relatively compact in \(\fP^q (\fP^q (\Omega))\) and every accumulation point \(Q^0\) is in \(\ocA^0 (\nu^0)\).
	Thus, by \cite[Proposition~A.1]{LakSPA15}, we get that
	\[
	E^{Q^0} \big[ \psi \big] = \lim_{n \to \infty} E^{Q^n} \big[ \psi \big] = \sup_{Q \in \ocA^0 (\nu^0)} E^Q \big[ \psi \big].
	\]
	This is the claim.
	
	\smallskip
	(ii). By Theorem \ref{theo: main1} (iv), there exists a sequence \((Q^n)_{n = 1}^\infty\) such that \(Q^n \in \ocA^n (\nu^n)\) and \(Q^n \to Q^0\) in \(\fP^q (\fP^q (\Omega))\). Using that \(Q^0\) is optimal, Theorem~\ref{theo: main1}~(vi) and \cite[Proposition~A.1]{LakSPA15}, we obtain that
	\[
	\lim_{n \to \infty} \sup_{Q \in \ocA^n (\nu^n)} E^{Q} \big[ \psi \big] = E^{Q^0} \big[ \psi \big] = \lim_{n \to \infty} E^{Q^n} \big[ \psi \big].
	\]
	Consequently, 
	\[
	0 \leq \varepsilon^n := \sup_{Q \in \ocA^n (\nu^n)} E^{Q} \big[ \psi \big] - E^{Q^n} \big[ \psi \big]\to 0,
	\]
	which shows that \(Q^n\) is an \(n\)-state \(\varepsilon^n\)-optimal control. The claim is proved.
\end{proof}

\begin{remark} \label{rem: other control frameworks}
	As we explain in Remark~\ref{rem: main1}~(ii), Theorem~\ref{theo: main1}, and hence also Corollary~\ref{coro: control}, also hold for the weak and relaxed frameworks. Moreover, using (a) and (b) from Remark~\ref{rem: main1}~(ii), one may also formulate versions of Corollary~\ref{coro: control} for strong control settings. For example, any sequences of \(n\)-state nearly optimal strong controls has an accumulation point that is optimal for the relaxed mean field system. 
\end{remark}

\section{Propagation of Chaos for \(G\)-Brownian motion} \label{app: G BM}
As discussed in part (ii) of Remark~\ref{rem: main1}, our framework is closely related to the idea of Peng's \(G\)-Brownian motion (see \cite{peng2007g} for an overview). In this short section, we explain the relation in more detail and deduce a propagation of chaos result for \(G\)-Brownian motion with drift interaction from Theorem~\ref{theo: main1}. We emphasize that our presentation does not aim for the most general result but for an illustration. 

\smallskip
Let $H:= \mathbb R$ and \(\CLM\) be the set of all probability measures \(P\) on \((\Omega, \cF, (\cF_t)_{t \in [0, T]})\), such that the coordinate process \(X\) is a continuous local \(P\)-martingale with absolutely continuous quadratic variation process \(\langle X, X \rangle\) with respect to $P$. 
Further, take two non-negative numbers \(a_* < a^*\) and set  
\[
\cR (x) := \Big\{ P \in \CLM \colon P (X_0 = x) = 1, \ (dt \otimes P)\text{-a.e. } d \langle X, X\rangle / dt \in [a_*, a^*] \Big\}.
\]

\smallskip
A \(G\)-Brownian motion is a family \(\{\mathcal{E}_x \colon x \in \bR\}\) of sublinear expectations on the space of upper semianalytic functions \(\Omega \to [- \infty, \infty]\) that is given by 
\[
\mathcal{E}_x (\varphi) := \sup_{P \in \cR (x)} E^P \big[ \varphi \big], 
\]
with \(\varphi \colon \Omega  \to [- \infty, \infty]\) upper semianalytic, i.e., \(\{\varphi > c\}\) is analytic for every \(c \in \bR\).

\smallskip
The set \(\cR (x)\) can be translated to our language of feedback controls. Let \(\cA (x)\) be the set of all laws of solution processes to the SDE 
\[
d Y_t = \sqrt{ \f_t (Y)} \, d W_t, \quad Y_0 = x, 
\]
where \(\f \colon [0, T] \times \Omega \to [a_*, a^*]\) is an arbitrary predictable process and \(W\) is a one-dimensional standard Brownian motion. It is clear that \(\cA (x) \subset \cR (x)\). Conversely, \cite[Proposition~2.4]{C23b} shows that \(\cR (x) \subset \cA (x)\). Hence, the \(G\)-Brownian motion has a representation as value function in a feedback control setting, namely 
\[
\cE_x (\varphi) = \sup_{P \in \cA (x)} E^P \big[ \varphi \big].
\]

Using the idea behind this feedback control representation, we can deduce a propagation of chaos result for a system of \(G\)-Brownian motions with drift interaction. 

\smallskip 
For \(n \in \mathbb{N}\), let \(\fPasn\) be the set of all probability measures \(P\) on \((\Omega^n, \cF^n, (\cF^n_t)_{t \in [0, T]})\) such that the coordinate process \(X = (X^1, \dots, X^n)\) is a continuous \(P\)-semimartingale with absolutely continuous semimartingale characteristics, whose densities we denote by \((b^P, c^P)\). 
Let \(h \colon \bR \to \bR\) be a bounded Lipschitz continuous function and define 
\begin{align*}
	\cR^n (x) := \Big\{ P \in \fPasn \colon P (X_0^i = x) = 1, \ 
& (dt \otimes P)\text{-a.e. } b^{P, i} = \frac{1}{n} \sum_{k = 1}^n h (X_{\cdot}^k),\\ &c^{P} \in \on{diag}^n, \ c^{P, ii} \in [a_*, a^*], \ i = 1, \dots, n \Big\},
\end{align*}
and
\begin{align*}
	\cR^0 (x) := \Big\{ P \in \mathfrak{S}^{\textup{ac}, 1}_\textup{sem} \colon P (X_0 = x) = 1, \
	(dt \otimes P)\text{-a.e. } b^P = E^P \big[ h (X_\cdot) \big],\ c^P \in [a_*, a^*] \Big\},
\end{align*}
where \(\on{diag}^n\) denotes the set of real-valued \(n \times n\) diagonal matrices.
Finally, for a bounded upper semianalytic function \(\varphi \colon \Omega \to \bR\), we define 
\begin{align*}
\mathcal{E}^n_x \circ \X_n^{-1} (\varphi) &:= \mathcal{E}^n_x \circ \X_n^{-1} (\mu \mapsto E^\mu[ \varphi ])
:= \sup_{P \in \cR^n (x)} \frac{1}{n} \sum_{k = 1}^n E^{P} \big[ \varphi (X^k) \big], \\ \mathcal{E}_x^0 (\varphi) &:= \sup_{P \in \cR^0 (x)} E^P \big[ \varphi \big].
\end{align*}
The family \(\{\mathcal{E}^n_x \circ \X_n^{-1} \colon x \in \bR\}\) consists of empirical distributions of \(G\)-Brownian motions with drift interaction, and \(\{\mathcal{E}^0_x \colon x \in \bR\}\) is a \(G\)-Brownian motion with McKean--Vlasov drift. 

\smallskip
We have the following propagation of chaos result:
\begin{theorem}
    For every bounded continuous function \(\varphi \colon \Omega \to \bR\), 
    \begin{align} \label{eq: conv G BM}
        \mathcal{E}^n_x \circ \X_n^{-1} (\varphi) \to \mathcal{E}^0_x (\varphi), \quad n \to \infty,
    \end{align}
    uniformly in \(x\) on compact subsets of \(\bR\). 
\end{theorem}

\begin{proof}
The strategy of proof is the following: we translate the statement into the feedback control framework from Section~\ref{sec: main1} and then deduce the claim from Theorem~\ref{theo: main1}.

\smallskip 
{\em Step 1: Translation.} 
Let \(\cA^n (x)\) be the set of all laws of solution processes \(Y = (Y^1, \dots, Y^n)\) to the SDE 
\begin{align*}
    d Y^k_t = \frac{1}{n} \sum_{i = 1}^n h (Y^i_t) \, dt + \sqrt{ \f^k_t (Y)} \, d W^k_t, \quad Y^k_0 = x, 
\end{align*}
where \(\f^1, \dots, \f^n \colon [0, T] \times \Omega^n \to [a_*, a^*]\) are arbitrary predictable processes and \(W^1, \dots, W^n\) are independent one-dimensional standard Brownian motions. Further, define \(\cA^0 (x)\) to be the set of all laws \(P\) of solution processes to the McKean--Vlasov SDE 
\[
d Y_t = E^P \big[ h (X_t) \big] \, dt + \sqrt{\f_t (Y)} \, d W_t, \quad Y_0 = x, 
\]
where \(\f \colon [0, T] \times \Omega \to [a_*, a^*]\) is an arbitrary predictable process and \(W\) is a one-dimensional standard Brownian motion. 

Clearly, \(\cA^n (x) \subset \cR^n (x)\) and \(\cA^0 (x) \subset \cR^0 (x)\). By virtue of the proof for \cite[Proposition~2.4]{C23}, also the converse inclusions hold and consequently, 
\[
\cA^n (x) = \cR^n (x), \quad \cA^0 (x) = \cR^0 (x).
\]
This implies that
\begin{align} \label{eq: trans}
	\sup_{P\in \cA^n (x)} \frac{1}{n} \sum_{k = 1}^n E^P \big[ \varphi (X^k) \big] = \mathcal{E}^n_x \circ \X_n^{-1} (\varphi), \qquad  
	\sup_{P \in \cA^0 (x)} E^P \big[ \varphi \big] = \mathcal{E}^0_x (\varphi).
\end{align}

{\em Step 2: Conclusion.}
The left hand expectations in \eqref{eq: trans} correspond to the framework from Section~\ref{sec: main1} with \(H \equiv \bR\), \(F \equiv [a_*, a^*],\) \(A \equiv 0,\) \(b (f, t, \omega, \mu) \equiv E^\mu[ h (X_t) ]\) and \(\sigma (f, t, \omega, \mu) \equiv \sqrt{f}\).
    It is not hard to check that the Conditions~\ref{cond: main1}, \ref{cond: compact assumption} and \ref{cond: main2} hold in this setting. Hence, using Theorem~\ref{theo: main1}~(vi) with \(\psi (\mu) \equiv E^\mu [ \varphi ]\) implies the claim. 

\smallskip
    To be precise, let us shortly explain that the r.h.s in \eqref{eq: conv G BM} coincides with the r.h.s. of \eqref{eq: compact convergence} that is given by 
    \[
    \sup \Big\{ E^Q\big[ \psi \big] \colon Q (\cA^0 (x)) = 1\Big\} = \sup \Big\{ \int E^\mu \big[ \varphi \big] Q (d \mu) \colon Q (\cA^0 (x)) = 1 \Big\}.
    \]
    First, we clearly have 
    \[
    \sup \Big\{ \int E^\mu \big[ \varphi \big] Q (d \mu) \colon Q (\cA^0 (x)) = 1 \Big\} \leq \sup_{P \in \cA^0 (x)} E^P \big[ \varphi \big] = \mathcal{E}^0_x (\varphi).
    \]
    On the other hand, Theorem~\ref{theo: main1}~(i) implies that the set \(\cA^0 (x)\) is nonempty and compact (in a suitable Wasserstein space) and hence, there exists a measure \(P^* \in \cA^0 (x)\) such that 
    \[
    \sup_{P \in \cA^0 (x)} E^P \big[ \varphi \big] = E^{P^*} \big[ \varphi  \big],
    \]
    where we use that \(P \mapsto E^P [ \varphi ]\) is continuous (in the weak and consequently, also in the stronger Wasserstein topology). 
    Now, for \(Q^* := \delta_{P^*}\), we get that
    \begin{align*}
    \mathcal{E}^0_x (\varphi) = \sup_{P \in \cA^0 (x)} E^P \big[ \varphi  \big] = \int E^\mu \big[ \varphi \big] Q^* (d \mu) \leq  \sup \Big\{ \int E^\mu \big[ \varphi \big] Q (d \mu) \colon Q (\cA^0 (x)) = 1 \Big\}.
    \end{align*}
    Consequently, the sublinear expectations on the r.h.s. of \eqref{eq: compact convergence} and \eqref{eq: conv G BM} coincide.
\end{proof}

	\section{Proof of Theorem \ref{theo: main1}} \label{sec: pf}
	
	In this section we prove our main Theorem \ref{theo: main1}. We start with some technical preparations, connecting our setting to relaxed controls. Thereafter, we proceed with the proof of Theorem~\ref{theo: main1} in a chronological order.  
	
	\subsection{A First Step Towards Relaxed Control Rules} \label{sec: relaxed setting}

	Let \(\m ([0, T] \times F)\) be the set of all Radon measures on \([0, T] \times F\) and define \(\m\) as its subset of all measures in \(\m ([0, T] \times F)\) whose projections on \([0, T]\) coincide with the Lebesgue measure. 
   We endow \(\m\) with the vague (equivalently, weak) topology, which turns it into a compact metrizable space (\cite[Theorem~2.2]{EKNJ88}). The Borel \(\sigma\)-field on \(\m\) is denoted by \(\mathcal{M}\) and the identity map on \(\m\) is denoted by \(M\).
		Further, we define the \(\sigma\)-fields 
		\[
		\mathcal{M}_t := \sigma \big(M (C) \colon C \in \mathcal{B}([0, t] \times F)\big), \quad t \in [0, T].
		\]
		On the product space \(\Theta:= \Omega \times \m\) we work with the product \(\sigma\)-field \(\mathcal{O} := \cF \otimes \mathcal{M}\) and the product filtration \(\mathbf{O} := (\mathcal{O}_t)_{t \in [0, T]}\) given by \(\mathcal{O}_t := \cF_t \otimes \mathcal{M}_t\). With little abuse of notation, we denote the coordinate map on \(\Theta\) by \((X, M)\).
		
		For \(g \in C^2_c (\bR; \bR), y^* \in D (A^*)\) and \((f, t, \omega, \nu) \in F \times [0, T] \times \Omega \times \W^\o (\Omega)\), we set 
		\begin{align*}
		\mathcal{L}_{g, y^*} (f, t, \omega, \nu) := g' (\langle \omega (t), y^*\rangle_H) \big( \langle \omega (t)&, A^* y^* \rangle_H + \langle b (f, t, \omega, \nu), y^*\rangle_H \big) 
		\\&+ \tfrac{1}{2} g'' (\langle \omega (t), y^* \rangle_H) \| \sigma^* (f, t, \omega, \nu) y^* \|^2_U.
		\end{align*}
	Finally, for \(g \in C^2_c (\bR; \bR), f \in F, \omega^1, \dots, \omega^n \in \Omega, y^1, \dots, y^n \in D (A^*), \nu \in \W^\o (\Omega)\) and \(i = 1, \dots, n\), we set 
	\begin{align*}
		\mathcal{L}^i_{g, y^1, \dots, y^n} (f, t&, \omega^1, \dots, \omega^n, \nu) 
		\\&:= 
		g' \Big( \sum_{k = 1}^n \langle \omega^k (t), y^k \rangle_H \Big) \big( \langle \omega^i (t), A^* y^i \rangle_H + \langle b (f, t, \omega^i, \nu), y^i \rangle_H \big) 
		\\&\hspace{4cm}+ \tfrac{1}{2} g'' \Big( \sum_{k = 1}^n \langle \omega^k (t), y^k \rangle_H \Big) \| \sigma^* (f, t, \omega^i, \nu) y^i \|^2_U.
	\end{align*}
		As we will show below, the following two definitions are equivalent to a relaxed control framework.

		\begin{definition} \label{def: C MK}
			For \(\nu \in \W^p (H)\), we define \(\cC^0 (\nu)\) to be the set of all measures \(Q \in \fP(\Theta)\) 
		with the following properties:
		\begin{enumerate}
			\item[\textup{(i)}] \(Q \circ X^{-1} \in \W^p (\Omega)\);
			\item[\textup{(ii)}] there exists an \(F\)-valued \(\mathbf{O}\)-predictable process \(\xi\) such that, for all \(g \in C^2_c (\bR; \bR)\) and \(y^* \in D(A^*)\), \(Q\)-a.s.
			\[
				\int_0^\cdot \mathcal{L}_{g, y^*} (\xi_s, s, X, Q^X_s) \, ds = \int_0^\cdot \int \mathcal{L}_{g, y^*} (f, s, X, Q^X_s) \, M (ds, df),
			\]
			where \(Q^X_s := Q \circ X^{-1}_{\cdot \wedge s}\);
			\item[\textup{(iii)}] possibly on a standard extension of \((\Theta, \mathcal{O}, \mathbf{O}, Q)\), there exists a cylindrical standard Brownian motion \(W\) such that a.s., for all \(t \in [0, T]\),
\begin{align*}
    X_t = S_t X_0 &+ \int_0^t S_{t - s} b (\xi_s, s, X, Q^X_s) \, ds 
    + \int_0^t S_{t - s} \sigma (\xi_s, s, X, Q^X_s) \, d W_s, 
\end{align*}
and \(X_0 \sim \nu\).
			\end{enumerate}
	Furthermore, we set 
	\[
	\cR^0 (\nu) := \big\{ P \in \fP (\fP(\Theta)) \colon P (\cC^0 (\nu)) = 1\big\}.
	\]
	\end{definition}

For \(n \in \mathbb{N}\), define
\[
\Y_n\colon \Theta^n \to \fP (\Theta), \quad \Y_n (\theta^1, \dots, \theta^n) := \frac{1}{n} \sum_{k = 1}^n \delta_{\theta^k}.
\]
Compared to \(\X_n\) from \eqref{eq: X_n def}, the empirical distribution \(\Y_n\) allows us to incorporate the particle systems together with their controls. 

		\begin{definition}\label{def: Cn}
	For \(n \in \mathbb{N}\) and \(\nu \in \W^p (H)\), let \(\cC^n (\nu)\)  be the set of all \(Q \in \fP (\Theta^n)\) with the following properties:
	\begin{enumerate}
		\item[\textup{(i)}] there exist  \(F\)-valued \(\mathbf{O}^n\)-predictable processes \(\xi^1, \dots, \xi^n\) such that, for all \(k = 1, \dots, n,\) \(g \in C^2_c (\bR; \bR)\) and \(y^1, \dots, y^n \in D (A^*)\), \(Q\)-a.s.
		\begin{align*}
			\int_0^\cdot \mathcal{L}^k_{g, y^1, \dots, y^n} &(\xi_s^k, s, X, \X_n (X_{\cdot \wedge s})) \, ds 
			\\&= \int_0^\cdot \int \mathcal{L}^k_{g, y^1, \dots, y^n} (f, s, X, \X_n (X_{\cdot \wedge s})) \, M^k (ds, df);
		\end{align*}
		\item[\textup{(ii)}] possibly on a standard extension of \((\Theta^n, \mathcal{O}^n, \mathbf{O}^n, Q)\), there exist independent cylindrical standard Brownian motions \(W^1, \dots, W^n\) such that, for \(k = 1, \dots, n\), a.s., for all \(t \in [0, T]\),
		\begin{align*}
			X^k_t = S_t X^k_0 &+ \int_0^t S_{t - s} b (\xi^k_s, s, X^k, \X_n (X_{\cdot \wedge s})) \, ds 
			\\&+ \int_0^t S_{t - s} \sigma (\xi^k_s, s, X^k, \X_n (X_{\cdot \wedge s})) \, d W^k_s,
		\end{align*}
		and \((X^1_0, \dots, X^n_0) \sim \bigotimes_{1}^n \nu\).\footnote{That means, \(X^1_0, \dots, X^n_0\) are i.i.d. with distribution \(\nu\).}
	\end{enumerate}
	Further, we define 
	\[
	\cR^n (\nu) := \big\{ Q \circ \Y_n^{-1} \colon Q \in \cC^n (\nu) \big\} \subset \fP (\fP (\Theta)).
	\]
\end{definition}

	\subsection{Some Moment Estimates}
We start with moment estimates that follow from the linear growth Condition \ref{cond: main1} (ii) by a standard Gronwall argument. 
\begin{lemma} \label{lem: moment estimates} 
	Assume that Condition \ref{cond: main1} \textup{(ii)} holds and let \(K \subset \W^p (H)\) be a bounded set in the sense that \(\sup_{\nu \in K} \int \|x\|^p_H \, \nu (dx) < \infty\). Then, 
		\begin{align} \label{eq: moment estimate lemma i}
		\sup \Big\{ \frac{1}{n} \sum_{k = 1}^n E^Q \big[ \|X^k\|_T^p \big] \colon Q \in \cC^n (\nu), n \in \mathbb{N}, \nu \in K \Big\} < \infty,
		\end{align}
	and
		\begin{align} \label{eq: moment estimate lemma ii}
		\sup \Big\{ E^Q \big[ \|X\|_T^p \big] \colon Q \in \cC^0 (\nu), \nu \in K \Big\} < \infty.
		\end{align}
\end{lemma}

\begin{proof}
	We start with \eqref{eq: moment estimate lemma i}.
	Take \(n \in \mathbb{N}, \nu \in K\) and \(Q \in \cC^n (\nu)\). For \(\scl > 0\), define
	\[
	T_\scl := \inf \Big\{ t \in [0, T] \colon \frac{1}{n} \sum_{k = 1}^n \|X^k\|_t^p \geq \scl \,\Big\}
	\]
	with the usual convention \(\inf \emptyset := \infty\).
	Thanks to Condition \ref{cond: main1} (ii), for all \(f \in F, t \in [0, T], \omega = (\omega^1, \dots, \omega^n)\) and \(1 \leq k \leq n\), we obtain that
	\begin{align*}
		\|b (f, t, \omega^k, \X_n (\omega (\cdot \wedge t)))\|_H^p & \leq C \Big[ 1 + \|\omega^k\|_t^p + \|\X_n (\omega (\cdot \wedge t))\|^p_p \Big] 
		\\&= C \Big[1 + \|\omega^k\|_t^p + \frac{1}{n} \sum_{i = 1}^n \|\omega^i\|_t^p \Big],
	\end{align*}
	where the constant depends on \(p\) and the linear growth constant from Condition~\ref{cond: main1}~(ii). Similarly, we get that 
	\begin{align*}
		\|S_s \sigma (f, t, \omega^k, \X_n (\omega (\cdot \wedge t)))\|_{L_2 (U, H)} \leq \j (s) \Big[ 1+  \|\omega^k\|_t + \Big(\frac{1}{n} \sum_{i = 1}^n \|\omega^i\|^p_t \Big)^{1/p} \, \Big].
	\end{align*}
	Using these bounds and \cite[Lemma 4.2]{C23} (this lemma requires the integrability condition~\eqref{eq: DZ cond}), there exists a constant \(C > 0\) that only depends on \(\j, \alpha, p, T, K\), the constant from Condition~\ref{cond: main1} (ii), and  \(\sup_{s \in [0, T]} \| S_s\|_{L(H)}\),\footnote{\(\sup_{s \in [0, T]} \| S_s\|_{L(H)}\) is finite by \cite[Proposition~I.5.5]{EN}} such that 
	\begin{align*}
		\frac{1}{n} \sum_{k = 1}^n E^Q \big[ \|X^k\|_{t \wedge T_\scl}^p \big] &\leq \frac{C}{n} \sum_{k = 1}^n \Big( 1 + E^Q \Big[ \int_0^{t \wedge T_\scl} \|b ( \xi^k_s, s, X^k, \X^n (X_{\cdot \wedge s})) \|^p_H \, ds \Big] 
		\\&\hspace{1cm}+ E^Q\Big[ \sup_{s \in [0, t \wedge T_\scl]} \Big\| \int_0^s S_{t - r} \sigma (\xi^k_r, r, X^k, \X^n (X_{\cdot \wedge r})) \, d W^k_r \Big\|_H^p \Big] \Big)
		\\&\leq \frac{C}{n} \sum_{k = 1}^n \Big( 1 + E^Q \Big[ \int_0^{t} \Big( \|X^k\|^p_{s \wedge T_\scl} + \frac{1}{n} \sum_{i = 1}^n \|X^i\|^p_{s \wedge T_\scl} \Big) \, ds \Big] \Big)
		\\&= C \Big( 1 + \int_0^t \frac{1}{n} \sum_{k = 1}^n E^Q \big[  \|X^k\|^p_{s \wedge T_\scl} \big] ds \Big).
	\end{align*}
	Thanks to Gronwall's and Fatou's lemma (using the fact that \(T_{\scl} \to \infty\) as \(\scl \to \infty\)), it follows that 
	\[
	\frac{1}{n} \sum_{k = 1}^n E^Q \big[ \|X^k\|_{T}^p \big] \leq C.
	\]
	By the dependencies of the constant \(C\), we conclude that \eqref{eq: moment estimate lemma i} holds.
	
	Next, we explain \eqref{eq: moment estimate lemma ii}.
	Recall that \(Q \circ X^{-1} \in \W^p (\Omega)\) for all \(Q \in \cC^0 (\nu)\) by definition, and notice that
	\[
	\|Q^X_t\|^p_p = \int \|\omega\|_T^p \, Q^X_t (d\omega) = E^Q \big[ \|X\|^p_t \big].
	\]
	With this observation at hand, \eqref{eq: moment estimate lemma ii} follows from Gronwall's lemma along the same lines as \eqref{eq: moment estimate lemma i} above. We omit the details for brevity. 
\end{proof}

\subsection{Martingale Problem Characterizations of \(\cC^0\) and \(\cC^n\)} \label{sec: rel con}
  
	In the following, we provide martingale problem characterizations for the sets \(\cC^0 (\nu)\) and \(\cC^n (\nu)\). 
	 Let \(\mathcal{D} (A^*)\) be a countable subset of \(D (A^*)\) that is dense in the graph norm on \(D(A^*)\). Such a set exists as \(A\) generates a strongly continuous semigroup, see \cite[Lemma 7.3]{C23} for details. Further, let \(\mathcal{C}^2_c\) be a countable subset of \(C^2_c(\bR; \bR)\) that is dense for the norm \(\|f\|_\infty + \|f'\|_\infty + \|f''\|_\infty\). Finally, for \(s \in [0, T]\), let \(\mathcal{T}_s \subset C_b (\Theta; \bR)\) be a countable separating class for \(\mathcal{O}_s\). The existence of such a class follows as in the proof of \cite[Lemma A.1]{LakSIAM17}.
	\begin{lemma} \label{lem: mg chara}
		Suppose that Condition \ref{cond: main1} holds. Let \(\nu \in \W^p (H)\) and \(Q\in \fP(\Theta)\). The following are equivalent:
		\begin{enumerate}
			\item[\textup{(i)}]
			\(Q \in \cC^0 (\nu)\).
			\item[\textup{(ii)}] The following properties hold:
			\begin{enumerate}
			\item[\textup{(a)}]
			\(Q \circ X^{-1}_0 = \nu\);
			\item[\textup{(b)}] \(Q \circ X^{-1} \in \W^p(\Omega)\);
			\item[\textup{(c)}] 
			for all \(y^* \in D(A^*)\) and \(g \in C^2_c (\bR; \bR)\), the process 
			\[
			\mathsf{M}^{g, y^*} := g (\langle X, y^*\rangle_H) - \int_0^\cdot \int \mathcal{L}_{g, y^*} (f, s, X, Q^X_s) \, M (ds, df)
			\]
			is a (square integrable) \(Q\)-\(\mathbf{O}\)-martingale.
			\end{enumerate}
				\item[\textup{(iii)}] The following properties hold:
			\begin{enumerate}
				\item[\textup{(a)}]
				\(Q \circ X^{-1}_0 = \nu\);
				\item[\textup{(b)}] \(Q \circ X^{-1} \in \W^p(\Omega)\);
				\item[\textup{(c)}] 
				for all \(y^* \in \mathcal{D} (A^*), g \in \mathcal{C}^2_c, s, t \in \mathbb{Q}_+ \cap [0, T], s < t\) and all \(\mathfrak{t} \in \mathcal{T}_s\), 
				\[
				E^Q \big[ (\mathsf{M}^{g, y^*}_t - \mathsf{M}^{g, y^*}_s) \mathfrak{t} \big] = 0.
				\]
		\end{enumerate}
		\end{enumerate}
	\end{lemma}
\begin{proof}
	We will prove the following implications:
	\begin{align*}
		\textup{(i)} \Rightarrow \textup{(ii)}, \quad \textup{(ii)} \Rightarrow \textup{(i)}, \quad \textup{(iii)} \Rightarrow \textup{(ii)}.
	\end{align*}
As (ii) \(\Rightarrow\) (iii) is trivial, these implications complete the proof. 
	
{\em \(\textup{(i)} \Rightarrow \textup{(ii)}\):} 
In the following we work on a standard extension of \((\Theta, \mathcal{O}, \mathbf{O}, Q)\). Further, we use the notation from Definition \ref{def: C MK} (ii).
Let us pass to the analytically weak formulation of the controlled SPDE from Definition \ref{def: C MK}. Namely, by \cite[Theorem 13]{MR2067962}, for every \(y^* \in D(A^*)\), a.s.
			\begin{align*}
					\langle X, y^*\rangle_H = \langle X_0, y^*\rangle_H &+ \int_0^\cdot \big(\langle X_s, A^* y^* \rangle_H  + \langle b (\xi_s, s, X, Q^X_s), y^* \rangle_H \big) \, ds 
					\\&+ \int_0^\cdot \langle \sigma^* (\xi_s, s, X, Q^X_s) y^*, d W_s \rangle_U.
				\end{align*}
Thus, It\^o's formula yields that a.s.
\begin{align*}
	g (\langle X, y^*\rangle_H) &- \int_0^\cdot \mathcal{L}_{g, y^*} (\xi_s, s, X, Q^X_s) \, ds 
	\\&= g (\langle X_0, y^* \rangle_H) + \int_0^\cdot g' ( \langle X_s, y^*\rangle_H ) \langle \sigma^* (\xi_s, s, X, Q^X_s) y^*, d W_s \rangle_U.
\end{align*}
Let \([\,\cdot\, ]\) be the quadratic variation process. Then, we obtain that a.s.
\begin{align*}
\Big[ g (\langle X, y^*\rangle_H) &- \int_0^\cdot \mathcal{L}_{g, y^*} (\xi_s, s, X, Q^X_s) \, ds \Big]_T 
\\&= \int_0^T \big( g' (\langle X_s, y^*\rangle_H) \big)^2 \|\sigma^* (\xi_s, s, X, Q^X_s) y^*\|^2_U \, ds.
\end{align*}
Thanks to the linear growth conditions from Condition \ref{cond: main1} (ii), it follows that 
\begin{align*}
	E^Q \Big[ \Big[ g (\langle X, y^*\rangle_H) - \int_0^\cdot \mathcal{L}_{g, y^*} (\xi_s, s, X, Q^X_s) \, ds \Big]_T \Big] < \infty.
\end{align*}
Hence, the process
\[
g (\langle X, y^*\rangle_H) - \int_0^\cdot \mathcal{L}_{g, y^*} (\xi_s, s, X, Q^X_s) \, ds
\]
is a (square integrable) \(Q\)-\(\mathbf{O}\)-martingale. By Definition \ref{def: C MK} (ii), this process coincides \(Q\)-a.s. with \(\mathsf{M}^{g, y^*}\). Consequently, (ii) follows. 

{\em \(\textup{(ii)} \Rightarrow \textup{(i)}\):} 
	It is known (see, e.g., \cite[Lemma 3.2]{LakSPA15}) that there exists a \(\mathbf{O}\)-predictable probability kernel \(\v\) from \([0, T] \times \Theta\) into \(F\) such that 
	\[
	M (dt, df) = \v (t, M, df) \, dt.
	\]
	As \(Q \circ X^{-1} \in \W^p(\Omega)\), the map \(t \mapsto Q^X_t\) is continuous from \([0, T]\) into \(\W^p (\Omega)\).
	Hence, with Condition~\ref{cond: main1}~(i), we get that the map 
	\[
	(f, t, \omega) \mapsto \mathfrak{L} (f, t, \omega) := \big( \mathcal{L}_{g, y^*} (f, t, \omega, Q^X_t) \big)_{g\hspace{0.02cm} \in\hspace{0.02cm} \mathcal{C}^2_c,\, y^*\hspace{0.02cm} \in\hspace{0.02cm} \mathcal{D} (A^*)}
	\]
	is continuous, where the image space is endowed with the product topology. Notice also that \((t, \omega, m) \mapsto \mathfrak{L} (f, t, \omega)\) is \(\mathbf{O}\)-predictable for every \(f \in F\) (cf. \cite[Theorem IV.97]{dellacheriemeyer}).
	Further, by (iii) from Condition~\ref{cond: main1}, the set 
	\[
	\Lambda(t, \omega) := \big\{ \mathfrak{L} (f, t, \omega) \colon f \in F \big\}
	\]
	is convex. Hence, \cite[Theorems II.4.3, II.6.2]{sion} yield that, for all \((t, \omega, m) \in [0, T] \times \Theta\),
	\[
\pi (t, \omega, m) := \int \mathfrak{L}(f, t, \omega) \, \v (t, m, df) \in \Lambda(t, \omega).
	\]
	Notice that \(\pi\) is \(\mathbf{O}\)-predictable.
	We deduce from Filippov's implicit function theorem (\cite[Theorem 18.17]{charalambos2013infinite}) that there exists an \(F\)-valued \(\mathbf{O}\)-predictable process \(\xi\) such that \(\pi (t, \omega, m) = \mathfrak{L} (\xi_t (\omega, m), t, \omega)\) for all \((t, \omega, m) \in [0, T] \times \Theta\). In particular, by a density argument, \(\xi\) is as in Definition \ref{def: C MK} (ii), i.e., we can replace \(\mathcal{C}^2_c\) by \(C^2_c(\bR; \bR)\) and \(\mathcal{D} (A^*)\) by \(D (A^*)\).
	In summary, for all \(g \in C^2_c (\bR; \bR)\) and \(y^* \in D (A^*)\), the processes 
	    \[
	g (\langle X, y^*\rangle_H) - \int_0^\cdot \mathcal{L}_{g, y^*} (\xi_s, s, X, Q^X_s) \, ds
	\]
	are \(Q\)-\(\mathbf{O}\)-martingales.
	Now, we may conclude that \(Q \in \cC^0 (\nu)\) from a standard representation theorem for cylindrical local martingales (\cite[Theorem~3.1]{OndRep}) and the equivalence of the analytical weak and mild formulation (\cite[Theorem~13]{MR2067962}), see Step 5 of the proof of \cite[Theorem 2.5]{C23} for details. 
	
	{\em \(\textup{(iii)} \Rightarrow \textup{(ii)}\):} This implication follows readily by a density argument. We omit the details for brevity.
\end{proof}

A similar result can also be proved for the set \(\cC^n (\nu)\). 
	\begin{lemma} \label{lem: mg chara N}
	Suppose that Condition \ref{cond: main1} holds. Let \(n \in \mathbb{N}, \nu \in \W^p (H)\) and \(Q\in\fP(\Theta^n)\). The following are equivalent:
	\begin{enumerate}
		\item[\textup{(i)}]
		\(Q \in \cC^n (\nu)\).
		\item[\textup{(ii)}] The following hold:
		\begin{enumerate}
			\item[\textup{(a)}]
			\(Q \circ (X^1_0, \dots, X^n_0)^{-1} = \bigoplus_1^n \nu\);
			\item[\textup{(b)}]
			for all \(y^1, \dots, y^n \in D (A^*)\) and \(g \in C^2_c (\bR; \bR)\), the process 
			\[
			g \Big( \sum_{k = 1}^n\, \langle X^k, y^k\rangle_H\Big) - \sum_{k = 1}^n \int_0^\cdot \int \mathcal{L}^k_{g, y^1, \dots, y^n} (f, s, X, \X_n (X_{\cdot \wedge s})) \, M^k (ds, df)
			\]
			is a (square integrable) \(Q\)-\(\mathbf{O}\)-martingale.
		\end{enumerate}
	\end{enumerate}
\end{lemma}

\begin{proof} The lemma follows similar to the proof of (i) \(\Leftrightarrow\) (ii) from Lemma \ref{lem: mg chara}. We omit the details for brevity. \end{proof}

We also relate the sets \(\cC^0 (\nu)\) and \(\cC^n(\nu)\) to \(\cA^0 (\nu)\) and \(\cA^n (\nu)\), respectively.

\begin{lemma} \label{lem: nonlinear SPDE rel RCR}
   Suppose that Condition~\ref{cond: main1} holds and take \(\nu \in \W^p (H)\). The following two equalities hold:
   \begin{enumerate}
          \item[\textup{(i)}] 
   \(\cA^0(\nu) = \{ Q \circ X^{-1} \colon Q \in \cC^0 (\nu)\}\).    
       \item[\textup{(ii)}]
   \(\cA^n (\nu) = \{ Q \circ (X^1, \dots, X^n)^{-1} \colon Q \in \cC^n (\nu)\}\).    
   \end{enumerate}
\end{lemma}
\begin{proof}
    (i). Suppose that \(P \in \cA^0 (\nu)\) and let \(\f\) be as in Definition \ref{def: cA0}. Then, the measure \(P \circ (X, \delta_{\f_t} (df) dt)^{-1}\) is an element of \(\cC^0 (\nu)\). Consequently, we have \[\cA^0(\nu) \subset \{ Q \circ X^{-1} \colon Q \in \cC^0 (\nu)\}.\]
    Conversely, assume that \(P = Q \circ X^{-1}\) for some \(Q \in \cC^0 (\nu)\). Recall the martingale characterization for \(\cC^0(\nu)\) that is given by Lemma~\ref{lem: mg chara}. We deduce from \cite[Theorem~9.19, Proposition~9.24]{jacod79} that, for all \(y^* \in D (A^*)\) and \(g \in C^2_c (\bR; \bR)\),  
    \[
    g (\langle X, y^*\rangle_H) - \int_0^\cdot E^{Q} \big[ \mathcal{L}_{g, y^*} (\xi_s, s, X, Q^X_s) \mid X^{-1}(\cF_{s-}) \big] \, ds
    \]
is a \(Q\)-\(X^{-1} (\mathbf{F})\)-martingale. Using Filippov's implicit function theorem similarly as in the proof for Lemma \ref{lem: mg chara} (with \((t, \omega) \mapsto Q (\xi_t \in df \mid X^{-1}(\cF_{t-} )) (\omega)\) instead of \((t, \omega, m) \mapsto \v (t, m, df)\)), we obtain the existence of a \(X^{-1}(\mathbf{F})\)-predictable process \(\f = \f \circ X\) such that \(Q\)-a.s.
\[
E^{Q} \big[ \mathcal{L}_{g, y^*} (\xi_s, s, X, Q^X_s) \mid X^{-1}(\cF_{s-}) \big] = \mathcal{L}_{g, y^*} (\f_s, s, X, Q^{X}_s), \quad g \in \mathcal{C}^2_c, \ y^* \in \mathcal{D} (A^*).
\]
A density argument shows that this equality holds for all \(g \in C^2_c(\bR; \bR)\) and \(y^* \in D(A^*)\). It follows from \cite[Theorem 10.37]{jacod79}, which is a general result dealing with the change of probability spaces, that 
    \[
    g (\langle X, y^*\rangle_H) - \int_0^\cdot \mathcal{L}_{g, y^*} (\f_s, s, X, Q^X_s) \, ds
    \]
    is a \(P\)-\(\mathbf{F}\)-martingale.
Finally, a standard representation theorem for cylindrical local martingales (\cite[Theorem~3.1]{OndRep}) and the relation of weak and mild solutions (\cite[Theorem~13]{MR2067962}) shows that \(Q \circ X^{-1} \in \cA^0 (\nu)\), see Step 5 of the proof of \cite[Theorem 2.5]{C23} for details. The proof of (i) is complete.

    (ii). This claim follows similar to (i). We omit a detailed proof for brevity.
\end{proof}

The following observation follows directly from Lemmata \ref{lem: mg chara N} and \ref{lem: nonlinear SPDE rel RCR}.

\begin{corollary} \label{coro: conv}
	Suppose that Condition~\ref{cond: main1} holds. For every \(\nu \in \W^p (H)\) and \(n \in \mathbb{N}\), the sets \(\cC^n (\nu)\) and \(\cA^n (\nu)\) are convex. 
\end{corollary}

\subsection{Compactness properties}
In this section we investigate (relative) compactness properties of the sets \(\cR^n, \cR^0, \cA^n\) and \(\cA^0\). Let \(\mathsf{r} \colon \m \times \m \to [0, 1]\) be a metric that induces the vague topology on \(\m\) and set 
\[
\mathsf{d} \colon \Theta \times \Theta \to \bR_+, \quad \mathsf{d} ((\omega^1, m^1), (\omega^2, m^2)) := \|\omega^1 - \omega^2\|_T + \mathsf{r} (m^1, m^2).
\]
We define the \(\o\)-Wasserstein space
\[
\W^q (\Theta) := \Big\{\mu \in \fP(\Theta) \colon \int \mathsf{d} (\theta, \theta_0)^\o \, \mu (d \theta) < \infty \Big\},
\]
where \(\theta_0 = (0, m_0) \in \Theta\) is a reference point. Similarly, we define \(\W^\o (\W^\o (\Theta))\), where we use the Wasserstein metric related to \(\mathsf{d}\) for \(\W^\o (\Theta)\). Of course, we endow \(\W^\o (\W^\o(\Theta))\) again with the corresponding \(\o\)-Wasserstein topology. 

\begin{lemma} \label{lem: Rk rel comp}
        Suppose that the Conditions \ref{cond: main1} \textup{(i) -- (ii)} hold. Let \(K \subset \W^p (H)\) be bounded in the sense of Lemma~\ref{lem: moment estimates} and relatively compact in \(\W (H)\). We define the set $\cR (K)$ by
        \[
		\cR (K) := \bigcup_{n \in \mathbb{N}} \bigcup_{\nu \in K} \cR^n (\nu).
		\]
        Then, under either Condition \ref{cond: compact assumption} or \ref{cond: iii alternative}, the set $\cR (K)$ is relatively compact in \(\W^\o (\W^\o (\Theta))\).
        
\end{lemma}

 Under Condition~\ref{cond: iii alternative}, our proof strategy is to apply Kolmogorov's tightness criterion to infer the relative compactness of $\cR(K)$. To apply it we need the following estimate that is proved below the proof of Lemma~\ref{lem: Rk rel comp}. 

\begin{lemma} \label{lem: difference moment estimates}
Suppose that the Conditions \ref{cond: main1} \textup{(i) -- (ii)} and Condition \ref{cond: iii alternative}  hold. Let \(K \subset \W^p (H)\) be bounded (in the sense as in Lemma~\ref{lem: moment estimates}) and define $\delta := p(1-\rho) / 2 -1>0 $. 
Then, there exists a constant $\C=\C(K,p,T) \in (0, \infty)$ such that
\begin{align} \label{eq: main estimate for KTC}
    \sup_{n\in\mathbb N}\sup_{P\in \cC^n(K)} \dfrac{1}{n}\sum_{k=1}^n E^P\big[\| (X^k_t - S_t X^k_0) - (X^k_s - S_s X^k_0)\|^p_H \big]\leq \C \hspace{0.05cm}|t-s|^{1+\delta}
\end{align}
for all $s, t \in [0, T]$.
\end{lemma}

\begin{proof}[Proof of Lemma \ref{lem: Rk rel comp}]
First, we show that 
\begin{align} \label{eq: 2nd part tightness}
\sup_{Q \in \cR (K)} \iint \mathsf{d} (\theta, \theta_0)^p \, \mu (d \theta) \, Q (d \mu) < \infty. 
\end{align}
For \(P \in \cC^n (\nu)\) with \(\nu \in K\), we obtain 
\begin{align*}
 \iint \mathsf{d} (\theta, \theta_0)^p\, \mu (d \theta) \, P \circ \Y^{-1}_n (d \mu) &= \frac{1}{n} \sum_{k = 1}^n E^P \Big[ \big(\|X^k\|_T + \mathsf{r} (M^k, m_0) \big)^p \Big] 
 \\&\leq \frac{2^{p - 1}}{n} \sum_{k = 1}^n E^P \big[ \|X^k\|_T^p \big] + 2^{p - 1}.
\end{align*}
Thanks to Lemma \ref{lem: moment estimates}, this estimate proves \eqref{eq: 2nd part tightness}.

\smallskip
For a measure \(\Q = P \circ \Y^{-1}_n \in \mathcal{R}^n (\nu)\), set
	\[
	\overline{Q} (G) := \frac{1}{n} \sum_{k = 1}^n P ( (X^k, M^k) \in G), \quad G \in \mathcal{O}.
	\]
If we show that \(\{\overline{Q} \colon Q \in \cR(K)\}\) is tight in \(\fP(\Theta)\) then, as \(p > \o\), by virtue of \eqref{eq: 2nd part tightness}, \cite[Corollary~B.2]{LakSPA15} implies that the set \(\cR (K)\) is relatively compact in \(\W^\o (\W^\o (\Theta))\).
    To do this, we divide the proof into two cases. In the first case, we  establish the tightness of \(\{\overline{Q} \colon Q \in \cR(K)\}\) under Condition \ref{cond: compact assumption}, and in the second under Condition \nolinebreak \ref{cond: iii alternative}.

    \smallskip 
    \textup{(i).} We prove the tightness of $\{\overline{Q} \colon Q \in \cR(K)\}$  under Condition  \ref{cond: compact assumption}, using the compactness method from \cite{gatarekgoldys}. First of all, since \(K \subset \W^p (H)\) is relatively compact in \(\W (H)\), it is tight. Hence, for every \(\varepsilon > 0\), there exists a compact set \(K'_\varepsilon \subset H\) such that 
    \[
    \sup_{\nu \in K} \nu ( H \setminus K'_\varepsilon) \leq \varepsilon.
    \] 
    From now on, we fix \(\varepsilon > 0\) and take the compact set \(K'_\varepsilon \subset H\) as above.
    For \(h \in L^p ([0, T]; H)\) and \(\lambda \in (1/p, 1]\), we set 
	\[
	R_\lambda h (t) := \int_0^t (t - s)^{\lambda - 1} S_{t - s} h (s) \, ds, \quad t \in [0, T], 
	\]
	where \((S_t)_{t \geq 0}\) is the semigroup generated by \(A\). Thanks to Condition~\ref{cond: compact assumption}, by \cite[Proposition~1]{gatarekgoldys}, \(R_\lambda\) is a compact operator from \(L^p ([0, T]; H)\) into \(\Omega\).
	For \(\ell > 0\), set 
	\begin{align*}
	K_\ell := \Big\{ \omega \in \Omega \colon \, &\omega  = S x + R_1 \psi + \tfrac{\sin(\pi \alpha)}{\pi} R_\alpha \phi, \\
	&x \in K'_\varepsilon, \, \psi, \phi \in L^p ([0, T]; H) \text{ with } \int_0^T \|\psi (s)\|^p_H \, ds \vee \int_0^T \|\phi (s)\|^p_H \, ds \leq \ell \, \Big\}.
	\end{align*}
By the compactness of the set \(K'_\varepsilon \subset H\) and the strong continuity of the semigroup \((S_t)_{t \geq 0}\), the set \(\{ S_\cdot x \colon x \in K'_\varepsilon\}\) is compact in \(\Omega\) by \cite[Lemma~1.5.2]{EN} and the Arzel\`a--Ascoli theorem (\cite[Theorem~A.5.2]{Kallenberg}). 
This together with the compactness of the operators \(R_1\) and \(R_\alpha\) entails relative compactness of \(K_\ell\) in \(\Omega\). 
Take \(P \in \cC^n (\nu)\), with \(\nu \in K\), and \(k \in \{1, \dots, n\}\).
The factorization formula (see Step 0 of the proof for \cite[Theorem 2.5]{C23} for a recap of the method) shows that \(P\)-a.s.
\begin{align} \label{eq: fak idenity}
X^k = S X^k_0 + R_1 ( s \mapsto b (\xi_s, s, X^k, \X_n (X_{\cdot \wedge s}))) + \tfrac{\sin(\pi \alpha)}{\pi} R_\alpha Y,
\end{align}
where 
\[
Y_t := \int_0^t (t - s)^{- \alpha} S_{t - s} \sigma (\xi^k_s, s, X^k, \X_n (X_{\cdot \wedge s})) \, d W^k_s, \quad t \in [0, T].
\]
Furthermore, by Eq. (4.4) from \cite{C23}, and Condition~\ref{cond: main1}~(ii), we have 
\begin{align} \label{eq: moment bound Y fak meth}
E^P \Big[ \int_0^T \|Y_s\|^p_H \, ds \Big] \leq C \Big( \int_0^T \Big[\frac{\j (s)}{s^{\alpha}} \Big]^2 \, ds \Big)^{p/2} E^P \Big[ \int_0^T \Big(1 + \|X^k\|^p_s + \frac{1}{n} \sum_{i = 1}^n \|X^i\|_s^p \Big) \, ds \Big].
\end{align}
Using \eqref{eq: fak idenity}, \eqref{eq: moment bound Y fak meth} and the definition of \(K_\ell\), it follows that
\begin{align*}
P (X^k \in K_\ell) &\geq 1 - \varepsilon - \frac{1}{\ell} \Big( E^P \Big[ \int_0^T \|b (\xi^k_s, s, X^k, \X_n (X_{\cdot \wedge s})\|_H^p \, ds \Big] + E^P \Big[ \int_0^T \|Y_s\|^p_H \, ds \Big] \Big)
\\&\geq 1 - \varepsilon - \frac{C}{\ell} E^P \Big[ \int_0^T \Big(1 + \|X^k\|^p_s + \frac{1}{n} \sum_{i = 1}^n \|X^i\|_s^p \Big) \, ds \Big],
\end{align*}
where we again used Condition~\ref{cond: main1}~(ii).
Now, by the moment bound from Lemma \ref{lem: moment estimates}, there exists a constant \(C > 0\) such that 
\begin{equation*}\begin{split}
	\frac{1}{n} \sum_{k = 1}^n P (X^k \in K_\ell) &\geq 1 - \varepsilon - \frac{C}{\ell} - \frac{C}{\ell} \frac{1}{n} \sum_{k = 1}^n E^P \Big[ \|X^k\|_T^p + \frac{1}{n} \sum_{i = 1}^n \|X^i\|_T^p \Big] 
	\\&= 1 - \varepsilon - \frac{C}{\ell} - \frac{2 C}{\ell} \frac{1}{n} \sum_{k = 1}^n E^P \big[ \|X^k\|_T^p \big]
	\\&\geq 1 - \varepsilon - \frac{C}{\ell}.
\end{split} \end{equation*}
Consequently, we obtain that
\[
\overline{P \circ \Y^{-1}_n} (K_\ell \times \m) = \frac{1}{n} \sum_{k = 1}^n P ( (X^k, M^k) \in K_\ell \times \m) \geq 1 - \varepsilon - \frac{C}{\ell}.
\]
As \(\m\) is compact, taking \(\ell\) large enough and recalling that \(\varepsilon > 0\) was arbitrary, this estimate shows tightness of \(\{ \overline{Q} \colon Q \in \cR (K)\}\) in the space \(\fP (\Theta)\) and hence, tightness of \(\cR (K)\) in \(\fP(\fP (\Theta))\). 
This proves the statement in case Condition \nolinebreak\ref{cond: compact assumption} holds. 

\smallskip
\textup{(ii).} 
Next, we proceed to establish the tightness of \(\{\overline{Q} \colon Q \in \cR(K)\} \subset \fP (\Theta)\) under Condition \nolinebreak \ref{cond: iii alternative}. Let \(K'_\varepsilon \subset H\) be a compact set as in (i) above. It follows from \cite[Lemma~I.5.2]{EN} and the Arzel\`a--Ascoli theorem (\cite[Theorem~A.5.2]{Kallenberg}) that the set \(C'_\varepsilon := \{ S_t x \colon x \in K'_\varepsilon \}\) is relatively compact in~\(\Omega\). Hence, for \(Q \in \cR (K)\) with \(Q = P \circ \Y_n^{-1}\), we get that  
\[
\overline{Q} ( S X_0 \in C'_\varepsilon ) = \frac{1}{n} \sum_{k = 1}^n P (S X_0 \in C'_\varepsilon) \geq \nu ( K'_\varepsilon ) \geq 1 - \varepsilon, 
\] 
which entails tightness of the family \( \{ \overline{Q} \circ (S X_0)^{-1} \colon Q \in \cR (K) \}\) in \(\W(\Omega)\). 
Furthermore, since
\begin{align*}
    \sup_{Q\in\cR(K)} E^{\overline Q} \big[\| &(X_t - S_t X_0) -(X_s - S_s X_0) \|^p_H \big] 
    \\&= \sup_{n\in\mathbb N}\sup_{P\in \cC^n(K)} \dfrac{1}{n}\sum_{k=1}^n E^P\big[\| (X^k_t - S_t X^k_0) - (X^k_s - S_s X^k_0) \|^p_H \big],
\end{align*}
we obtain tightness of the set $\{\overline{Q} \circ (X - SX_0)^{-1} \colon Q \in \cR(K)\}\subset \fP (\Omega)$ from Lemma \ref{lem: difference moment estimates} and Kolmogorov's tightness criterion (\cite[Theorem 23.7]{Kallenberg}). As \((\omega, \omega') \mapsto \omega + \omega'\) is continuous from \(\Omega \times \Omega\) into \(\Omega\), we conclude that \( \{ \overline{Q} \circ X^{-1} \colon Q \in \cR (K)\}\) is tight in \(\W(\Omega)\). Finally, as the space \(\m\) is compact, this implies tightness of \(\{ \overline{Q} \colon Q \in \cR (K)\}\subset  \fP (\Theta)\) and hence, tightness of \(\cR (K)\) in \(\fP(\fP (\Theta))\). This proves the statement under the assumption of Condition \ref{cond: iii alternative}.
\end{proof}
Before we prove Lemma \ref{lem: difference moment estimates}, we present a simple but useful estimate for the Hilbert-Schmidt norm of a linear operator. 
\begin{lemma}\label{lem: hilbert schmidt estimate}
Let $(e_k)_{k = 1}^\infty \subset H$ be a Riesz basis. Then, there  is a constant $C>0$  such that for every \(L \in L (U, H)\), it holds that
\begin{align*}
    \|L\|^2_{L_2 (U, H)}\leq C\sum_{k = 1}^\infty\|L^*e_k\|^2_{U}.
\end{align*}
\end{lemma}
\begin{proof}
Let $(u_\ell)_{\ell = 1}^\infty$ be an orthonormal basis of $U$. Then, by \eqref{eq: Bessel estimation} and Fubini's theorem, it holds that
\begin{align*}
\|L\|_{L_2 (U, H)}^2 &=\sum_{\ell = 1}^\infty\| Lu_{\ell}\|^2_{H} \leq C\sum_{\ell = 1}^\infty\sum_{k = 1}^\infty|\langle Lu_{\ell},e_k\rangle_{H}|^2
\\& =C\sum_{k = 1}^\infty\sum_{\ell = 1}^\infty |\langle u_{\ell},L^*e_k\rangle_{U}|^2 = C\sum_{k = 1}^\infty\|L^*e_k\|^2_{U}.
\end{align*}
This yields the claimed estimate.
\end{proof}
\begin{proof}[Proof of Lemma \ref{lem: difference moment estimates}]
It follows from \eqref{eq: constants} that $p>2/(1-\rho)$ and consequently, 
\begin{align*}
    \delta = (p/2)(1-\rho)-1>\frac{2}{2(1-\rho)}(1-\rho)-1=0.
\end{align*}
In the following, we establish the estimate \eqref{eq: main estimate for KTC}, where we adapt an idea from the proof of \cite[Theorem 2.6]{bhatt1998interacting}. Fix an arbitrary measure \(P \in \cC^n (\nu)\) for $\nu\in K$. 
By Definition~\ref{def: Cn}, there exist  \(F\)-valued \(\mathbf{O}^n\)-predictable processes \(\xi^1, \dots, \xi^n\) and independent cylindrical standard Brownian motions \(W^1, \dots, W^n\) such that, for \(k = 1, \dots, n\) and $X=(X^1,\dots,X^n)$, we have a.s., for all \(t \in [0, T]\),
	\begin{align*}
		X^k_t = S_t X^k_0 &+ \int_0^t S_{t - s} b (\xi^k_s, s, X^k, \X_n (X_{\cdot \wedge s})) \, ds 
		+ \int_0^t S_{t - s} \sigma (\xi^k_s, s, X^k, \X_n (X_{\cdot \wedge s})) \, d W^k_s. 
	\end{align*}
We define the auxiliary processes $Y^k=(Y^k_t)_{t\geq 0}$ and $Z^k=(Z^k_t)_{t\geq 0}$ by
    \begin{align*}
         Z^k_t &:= \int_0^t S_{t - u} b (\xi^k_u, u, X^k, \X_n (X_{\cdot \wedge u})) \, du, \\
         Y^k_t &:= \int_0^t S_{t - u} \sigma (\xi^k_u, u, X^k, \X_n (X_{\cdot \wedge u})) \, d W^k_u.
    \end{align*}
Clearly, we have
\begin{align}\label{eq: compact_lemma_eq1}
    E^P\big[\| (X^k_t - S_t X^k_0) - (X^k_s - S_s X^k_0) \|^p_H\big]\leq  C\,\big(E^P \big[\|Y^k_t-Y^k_s\|^p_H \big]+ E^P \big[\|Z^k_t-Z^k_s\|^p_H \big]\big).
\end{align}
  Using \cite[Corollary 3.3.2]{gawa}, Lemma~\ref{lem: hilbert schmidt estimate} and Condition~\ref{cond: iii alternative}, we obtain that
\begin{align}
    E^P  \big[ & \|  Y^k_t -Y^k_s\|_{H}^p\big] \nonumber
    \\&\leq C\,\Big( E^P\Big[\Big(\int_0^s \|(S_{t-u}-S_{s-u}) \sigma (\xi^k_u, u, X^k, \X_n (X_{\cdot \wedge u}))\|^2_{L_2 (U, H)} \,du\Big)^{p/2}\Big] \nonumber 
    \\ 
    &\qquad\qquad +E^P\Big[\Big(\int_s^t \|S_{t-u}\sigma (\xi^k_u, u, X^k, \X_n (X_{\cdot \wedge u}))\|^2_{L_2 (U, H)} \, du\Big)^{p/2}\Big]\Big)\nonumber
    \\ 
    &\leq C\,\Big( E^P\Big[ \Big(\int_0^s \sum_{\ell=1}^{\infty} \|\sigma^*(\xi^k_u, u, X^k, \X_n (X_{\cdot \wedge u}))(S^*_{t-u}-S^*_{s-u})e_\ell \|^2_{U} \, du\Big)^{p/2} \Big]\nonumber
    \\ 
     &\qquad\qquad +E^P\Big[ \Big(\int_s^t \sum_{\ell =1}^{\infty} \|\sigma^*(\xi^k_u, u, X^k, \X_n (X_{\cdot \wedge u})) S^*_{t-u}e_\ell \|^2_{U} \, du\Big)^{p/2} \Big]\Big) \label{eq: compact_lemma_eq2}
     \\
    &\leq C\,\Big( E^P\Big[ \Big(\int_0^s \sum_{\ell=1}^{\infty} \mathsf{c}_\ell^2 \big|e^{-\lambda_\ell(t-u)}-e^{-\lambda_\ell(s-u)}\big|^2 \big(1+\|X^k\|_{u}^2+\| \X_n (X_{\cdot \wedge u}) \|^2_p \big) \, du\Big)^{p/2} \Big] \nonumber
    \\ 
    &\qquad\qquad + E^P\Big[ \Big(\int_s^t \sum_{\ell=1}^{\infty} \mathsf{c}_\ell^2 e^{-2\lambda_\ell(t-u)} \big(1+\|X^k\|_{u}^2+\| \X_n (X_{\cdot \wedge u}) \|^2_p\big) \, du\Big)^{p/2} \Big]\Big ), \nonumber
\end{align}
and, using \eqref{eq: Bessel estimation}, we get that
\begin{align}
   E^P  \big[ & \| Z^k_t - Z^k_s\|_{H}^p \big] \nonumber
     \\&\leq C\,\Big( E^P\Big[\Big(\int_0^s \|(S_{t-u}-S_{s-u})b(\xi^k_u, u, X^k, \X_n (X_{\cdot \wedge u}))\|^2_{H} \, du\Big)^{p/2}\Big] \nonumber
     \\ 
    &\qquad\qquad +E^P\Big[\Big(\int_s^t \|S_{t-u}b(\xi^k_u, u, X^k, \X_n (X_{\cdot \wedge u}))\|^2_{H} \, du\Big)^{p/2}\Big]\Big) \nonumber
    \\ 
    &\leq C\,\Big ( E^P\Big[ \Big(\int_0^s \sum_{\ell =1}^{\infty} \big|\langle b(\xi^k_u, u, X^k, \X_n (X_{\cdot \wedge u}),(S^*_{t-u}-S^*_{s-u})e_\ell\rangle_{H}\big|^2 \, du\Big)^{p/2} \Big] \nonumber
    \\
     &\qquad\qquad +E^P\Big[ \Big(\int_s^t \sum_{\ell =1}^{\infty} \big | \langle b(\xi^k_u, u, X^k, \X_n (X_{\cdot \wedge u})), S^*_{t-u}e_\ell \rangle_{H}\big |^2 \, du\Big)^{p/2} \Big]\Big) \label{eq: compact_lemma_eq3}
     \\ 
    &\leq C\,\Big( E^P\Big[ \Big(\int_0^s \sum_{\ell =1}^{\infty} \mathsf{c}_\ell^2 \big|e^{-\lambda_\ell(t-u)}-e^{-\lambda_\ell(s-u)} \big|^2 \big(1+\|X^k\|_{u}^2+\| \X_n (X_{\cdot \wedge u}) \|^2_p \big) \, du\Big)^{p/2} \Big] \nonumber
    \\ 
    &\qquad\qquad + E^P\Big[ \Big(\int_s^t \sum_{\ell =1}^{\infty} \mathsf{c}_\ell^2 e^{-2\lambda_\ell(t-u)} \big(1+\|X^k\|_{u}^2+\| \X_n (X_{\cdot \wedge u}) \|^2_p\big) \, du\Big)^{p/2} \Big]\Big). \nonumber
\end{align}
We define 
\begin{align*}
    \psi_1(u) :=\sum_{\ell =1}^\infty\mathsf{c}_\ell^2 \big|e^{-\lambda_\ell(t-u)} -e^{-\lambda_\ell(s-u)}\big|^2
    \quad\text{and} \quad
    \psi_2(u):=  \sum_{\ell =1}^{\infty} \mathsf{c}_\ell^2 e^{-2\lambda_\ell(t-u)},
\end{align*}
and obtain, with Lemma~\ref{lem: moment estimates}, that 
\begin{equation} \label{eq: compact_lemma_eq4}
\begin{split}
     \frac{1}{n}\sum_{k=1}^n E^P\Big[\Big(\int_0^s&\, \psi_1(u)(1+\|X^k\|_{u}^2+\| \X_n (X_{\cdot \wedge u}) \|^2_p) \, du\Big)^{p/2}\Big]\\
    &\leq C \, \Big(\int_0^s\psi_1(u) \, du\Big)^{p/2} \Big( 1 + \frac{1}{n}\sum_{k=1}^nE^P\big[\|X^k\|_{T}^p\big]\Big)\\ 
    &\leq C\, \Big(\int_0^s\psi_1(u) \, du\Big)^{p/2},
\end{split}
\end{equation}
and, similarly,
\begin{align} \label{eq: compact_lemma_eq5}
\begin{split}
    \frac{1}{n}\sum_{k=1}^n &E^P\Big[\Big(\int_s^t\psi_2(u)(1+\|X^k\|_{u}^2+\| \X_n (X_{\cdot \wedge u}) \|^2_p) \, du\Big)^{p/2}\Big] \leq C \, \Big(\int_s^t\psi_2(u) \, du\Big)^{p/2}.
\end{split}
\end{align}
Using \eqref{eq: compact_lemma_eq1} -- \eqref{eq: compact_lemma_eq5}, it follows that
\begin{equation}\label{eq: estimate for Xt-Xs}
	\begin{split}
    \dfrac{1}{n}\sum_{k=1}^n E^P &\big[\| (X^k_t- S_t X^k_0) - (X^k_s - S_s X^k_0)\|^p_H \big] 
    \\&\leq C\,\Big(\Big(\int_0^s\psi_1(u) \, du\Big)^{p/2}+\Big(\int_s^t\psi_2(u) \, du\Big)^{p/2}\Big).
\end{split}
\end{equation} 
Recall the elementary inequality 
\begin{align} \label{eq: ele ineq}
1-e^{-x}\leq 1 \wedge x \leq x^{\varepsilon}
\end{align} 
for \(x > 0\) and \(\varepsilon \in [0, 1]\).
By Fubini's theorem, Condition \ref{cond: iii alternative}, and using \eqref{eq: ele ineq} with $\varepsilon=(1-\rho)/2\in [0,1]$,  we obtain that
\begin{align}\label{eq: estimate psi1}
\begin{split}
    \int_0^s\psi_1(u) \, du 
    &= \frac{1}{2} \sum_{k = 1}^\infty \cCk^2 \lambda_k^{-1} \big(1 - e^{- 2\lambda_k s}\big) \big(1 - e^{- \lambda_k (t - s)}\big)^2 
    \\&\leq \frac{1}{2}\sum_{k = 1}^\infty\cCk^2\lambda_k^{-1} \big(1 - e^{-\lambda_k(t-s)}\big)^2
    \\&\leq\frac{1}{2}\sum_{k = 1}^\infty\cCk^2\lambda_k^{-1} \big(\lambda_k (t-s) \big)^{1-\rho}
    \\&\leq C\, |t-s|^{1-\rho}. \phantom \int
\end{split}
\end{align}
Similarly, using \eqref{eq: ele ineq} with $\varepsilon = (1-\rho)\in [0,1]$, we also get that
\begin{align}\label{eq: estimate psi2}
\begin{split}
    \int_s^t\psi_2(u) \, du &= \frac{1}{2}\sum_{k = 1}^\infty\cCk^2\lambda_k^{-1}\big(1 - e^{-2\lambda_k(t-s)}\big) 
    \\&\leq \sum_{k = 1}^\infty\cCk^2\lambda_k^{-1}\big(\lambda_k(t-s)\big)^{1-\rho}
    \\&\leq C\, |t-s|^{1-\rho}. \phantom \int 
\end{split}
\end{align}
Recalling that \(p\, (1 - \rho) / 2 = 1 + \delta\), and taking the estimates \eqref{eq: estimate for Xt-Xs}, \eqref{eq: estimate psi1} and \eqref{eq: estimate psi2} into consideration, we conclude that \eqref{eq: main estimate for KTC} holds. 
\end{proof}

\begin{lemma} \label{lem: C comp}
   Suppose that the Conditions \ref{cond: main1} \textup{(i) -- (ii)} and one of the Conditions~\ref{cond: compact assumption} and \ref{cond: iii alternative} hold.
	Let \(K \subset \W^p (H)\) be bounded (in the sense of Lemma~\ref{lem: moment estimates}) and compact in \(\W^q (H)\).  
	Then, the set
	\begin{align*}
	\cC^0 (K) &:= \bigcup_{\nu \in K} \cC^0 (\nu)
	\end{align*}
	is 
	compact in \(\W^\o (\Theta)\).
\end{lemma}
\begin{proof}
Relative compactness of \(\cC^0(K)\) follows from similar arguments as used in the proof of Lemma \ref{lem: Rk rel comp}. 
	We only detail the proof for the closedness of \(\cC^0 (K)\). 
	Take a sequence \((Q^n)_{n = 0}^\infty\subset \cC^0(K)\) such that \(Q^n \to Q^0\) in \(\W^\o (\Theta)\). In the following we show that \(Q^0\) satisfies the properties (iii.a) -- (iii.c) from Lemma \ref{lem: mg chara}.

 \smallskip
	As \(K\) is compact in \(\W^q(H)\) and \(Q \mapsto Q \circ X^{-1}_0\) is continuous from \(\W^q(\Theta)\) into \(\W^q(H)\) (by the continuity of \(\omega \mapsto \omega (0)\) from \(\Omega\) into \(H\) and \cite[Proposition~A.1]{LakSPA15}), there exists a measure \(\nu_0 \in K\) such that \(Q^0 \circ X_0^{-1} = \nu_0\). 
	 Hence, (iii.a) from Lemma \ref{lem: mg chara} holds with initial distribution \(\nu_0\).

 \smallskip
	Next, it follows from Fatou's lemma (and Skorokhod's coupling theorem) that 
	\begin{align} \label{eq: p moment bound limit}
		E^{Q^0} \big[ \|X\|_T^p \big] \leq \liminf_{n \to \infty} E^{Q^n} \big[ \|X\|_T^p \big] 
		\leq \sup \Big\{ E^Q \big[ \|X\|_T^p \big] \colon Q \in \cC^0 (\nu), \nu \in K \Big\}.
	\end{align}
The final term is finite by Lemma \ref{lem: moment estimates}. Thus, \(Q^0 \circ X^{-1} \in \W^p (\Omega)\), which means that part (iii.b) from Lemma \ref{lem: mg chara} holds.
	
	\smallskip 
	Finally, we show that (iii.c) holds. Take \(g \in C^2_c(\bR; \bR)\) and \(y^* \in D(A^*)\).
	For \((r, \omega, m, \mu) \in [0, T] \times \Theta \times \W^\o (\Theta)\), define 
	\begin{align} \label{eq: Mr}
	\oM_r (\omega, m, \mu) := g (\langle \omega (r), y^*\rangle_H) - \int_0^r \int \mathcal{L}_{g, y^*} (f, u, \omega, \mu^X_u) \, m (du, df).
	\end{align}
\begin{lemma} \label{lem: Mr cont}
	Suppose that Condition~\ref{cond: main1}~\textup{(i)} holds. Then, \(\oM_r \colon \Theta \times \fP^q (\Theta) \to \bR\) is continuous for every \(r \in [0, T]\).
\end{lemma}
\begin{proof}
Take \(r \in [0, T]\) and a sequence \((\omega^n, m^n, \mu^n)_{n = 0}^\infty \subset \Theta \times \fP^q (\Theta)\) with \[(\omega^n, m^n, \mu^n) \to (\omega^0, m^0, \mu^0)\] in \(\Theta \times \fP^q (\Theta)\). Notice that 
	\begin{align*}
		\big|\oM_r (\omega^n, m^n, \mu^n) &- \oM_r (\omega^0, m^0, \mu^0)\big|
		\\&\leq |g (\langle \omega^n(r), y^*\rangle_H) - g (\langle \omega^0 (r), y^*\rangle_H)| 
			\\&\hspace{1.5cm}+ \sup_{f \in F, u \in [0, r]} \big| \mathcal{L}_{g, y^*} (f, u, \omega^n, \mu^{n, X}_u) - \mathcal{L}_{g, y^*} (f, u, \omega^0, \mu^{0, X}_u)\big| \phantom \int
		\\&\hspace{1.5cm} + \Big| \int_0^r \int \mathcal{L}_{g, y^*} (f, u, \omega^0, \mu^{0, X}_u) \, (m^n - m^0) (du, df) \Big|
		\\&=: I_n + II_n + III_n.
	\end{align*}
 Further, notice that \((u, \mu) \mapsto \mu^{X}_u\) is continuous from \([0, T] \times \fP^q (\Theta)\) into \(\fP^q (\Omega)\). This follows, for instance, from \cite[Proposition~A.1]{LakSPA15}.
That \(I_n \to 0\) is obvious and \(III_n \to 0\) follows from Condition~\ref{cond: main1}~(i) and the fact that \(\m\) is endowed with the weak topology. 
Further, \(II_n \to 0\) follows from Condition~\ref{cond: main1} (i) and Berge's maximum theorem (\cite[Theorem~17.31]{charalambos2013infinite}). 
	\end{proof}
		We set			
	\[
	\mathsf{M}^{n} (X, M) := \oM (X, M, Q^n) = g (\langle X, y^*\rangle_H) - \int_0^\cdot \int \mathcal{L}_{g, y^*} (f, s, X, Q^n \circ X_{\cdot \wedge s}^{-1}) \, M (ds, df).
	\]
	The following lemma provides the main step of the proof.
	\begin{lemma} \label{lem: conv}
		For every \(t \in [0, T]\) and any bounded continuous function \(\psi \colon \Theta \to \bR\),
		\[
		E^{Q^n} \big[ \mathsf{M}^n_t \psi \big] \to E^{Q^0} \big[ \mathsf{M}^0_t \psi \big].
		\]
	\end{lemma}
\begin{proof}
	By Skorokhod's coupling theorem, on some probability space \((\Sigma, \mathcal{G}, P)\), there are \(\Theta\)-valued random variables \((X^0, M^0),\) \((X^1, M^1),\) \(\dots\) with laws \(Q^0, Q^1, \dots\) such that \(P\)-a.s. \((X^n, M^n) \to (X^0, M^0)\). 
	By Lemma~\ref{lem: Mr cont}, \(P\)-a.s. 
	\[\oM_t (X^n, M^n, Q^n) \psi (X^n, M^n) \to \oM_t (X^0, M^0, Q^0) \psi (X^0, M^0).\]
Using Condition \ref{cond: main1} (ii) and Lemma \ref{lem: moment estimates}, we observe that 
\begin{align} \label{eq: moment bound}
\sup_{n \in \mathbb{N}} E^P \Big[ \big | \mathsf{M}^n_t (X^n, M^n) \psi (X^n, M^n)\big|^{p / 2} \Big] \leq C \Big( 1 + \sup_{n \in \mathbb{N}} E^P \big[ \|X^n\|^p_T \big] \Big) < \infty.
\end{align}
Consequently, because \(p / 2 > 1\), Vitali's theorem yields the claim.
\end{proof}

Let \(0 \leq s < t \leq T\) and take \(\mathfrak{t} \in \mathcal{T}_s\). The Lemmata~\ref{lem: mg chara} and \ref{lem: conv} imply that 
\[
E^{Q^0} \big[ \big(\mathsf{M}^0_t - \mathsf{M}^0_s \big) \mathfrak{t} \big] = \lim_{n \to \infty} E^{Q^n} \big[ \big(\mathsf{M}^n_t - \mathsf{M}^n_s \big) \mathfrak{t} \big] = 0.
\]
We conclude that (iii.c) from Lemma \ref{lem: mg chara} holds. 

In summary, \(Q^0 \in \cC^0 (\nu^0) \subset \cC^0 (K)\). This implies that \(\cC^0 (K)\) is closed and therefore, the proof is complete.
\end{proof}

We record a final observation.
\begin{lemma} \label{lem: cCN compact}
Suppose that the Conditions \ref{cond: main1} \textup{(i) -- (ii)} and one of the Conditions~\ref{cond: compact assumption} and \ref{cond: iii alternative} hold.
	For every \(\nu \in \W^p (H)\) and \(n \in \mathbb{N}\), the sets \(\cC^n (\nu)\) and \(\cR^n (\nu)\) are nonempty and compact in \(\W^\o (\Theta^n)\) and \(\W^\o (\W^\o (\Theta))\), respectively.
\end{lemma}
\begin{proof}
That \(\cC^n (\nu)\) is nonempty follows from Theorem~\ref{theo: existence appendix} in the appendix.
Similar to the proof of Lemma~\ref{lem: Rk rel comp}, one proves that the set \(\cC^n (\nu)\) is relatively compact in \(\W^\o (\Theta^n)\). Further, a martingale problem argument (in the spirit of those presented in Lemma~\ref{lem: C comp} for the set \(\cC^0 (\nu)\)) shows that \(\cC^n (\nu)\) is closed in \(\W^\o (\Theta^n)\). We omit the details for brevity. In summary, \(\cC^n (\nu)\) is nonempty and compact. These claims transfer directly to \(\cR^n (\nu)\) by the continuity of \(P \mapsto P \circ \Y_n^{-1}\) from \(\W^\o (\Theta^n)\) into \(\W^\o (\W^\o (\Theta))\), cf. \cite[Proposition A.1]{LakSPA15}.
\end{proof}

	\subsection{Proof of Theorem \ref{theo: main1} (i)}
	Using \cite[Proposition~A.1]{LakSPA15}, we obtain continuity of the maps \(\pi_n \colon \W^\o (\Theta^n) \to \W^\o (\Omega^n)\) and \(\Pi \colon \W^\o (\W^\o (\Theta)) \to \W^\o (\W^\o (\Omega))\) given through \(\pi_n (P) := P \circ (X^1, \dots, X^n)^{-1}\)  and \(\Pi (Q) := Q \circ \pi^{-1}_1\). 
	For \(n \in \mathbb{N}\) and \(\nu \in \W^p (H)\), Lemma \ref{lem: nonlinear SPDE rel RCR} yields that 
	\begin{align*}
		\Pi (\cR^n (\nu)) &= \{ \Pi (Q) \colon Q \in \cR^n (\nu) \}
		\\&= \{ \Pi (P \circ \Y_n^{-1}) \colon P \in \cC^n (\nu)\}
		\\&= \{ P \circ (X^1, \dots, X^n)^{-1} \circ \X_n^{-1} \colon P \in \cC^n(\nu) \}		
		\\&= \{ Q \circ \X_n^{-1} \colon Q \in \cA^n (\nu) \}
		\\&= \mathcal{U}^n (\nu).
	\end{align*} 
As, by Lemma \ref{lem: cCN compact}, \(\cR^n (\nu)\) is nonempty and compact in \(\W^\o(\W^\o (\Theta))\), it follows from the continuity of \(\Pi\) that \(\mathcal{U}^n (\nu)\) is nonempty and compact in \(\W^\o (\W^\o(\Omega))\).
	Similarly, as 
	\[
	\cA^0 (\nu) = \pi_1 ( \cC^0 (\nu)), \qquad \cA^n (\nu) = \pi_n (\cC^n (\nu)),
	\]
	by Lemma \ref{lem: nonlinear SPDE rel RCR}, it follows that the sets \(\cA^0 (\nu)\) and \(\cA^n (\nu)\) are compact by Lemmata~\ref{lem: C comp} and \ref{lem: cCN compact}. 
	Further, \(\cA^n (\nu)\) is nonempty. Anticipating the following section, the claim \(\cA^0 (\nu) \not = \emptyset\) follows from Theorem~\ref{theo: main1}~(ii). 
	Finally, because of the compactness of \(\cA^0 (\nu)\) in \(\W^\o (\Omega)\), the set \(\mathcal{U}^0 (\nu) = \{Q \colon Q (\cA^0 (\nu)) = 1\}\) is compact in \(\fP (\W^\o (\Omega))\) (by \cite[Theorem 15.11]{charalambos2013infinite}) and \(\W^\o (\W^\o (\Omega))\), as these spaces induce the same topology on \(\mathcal{U}^0 (\nu)\). This completes the proof. 
\qed 

\subsection{Proof of Theorem \ref{theo: main1} (ii)}
Below, we prove a version of Theorem \ref{theo: main1} (ii) for the sets \(\cR^n\) and \(\cR^0\) instead of \(\mathcal{U}^n\) and \(\mathcal{U}^0\). The claim of Theorem \ref{theo: main1} (ii) will then follow through projection as in the proof of Theorem \ref{theo: main1} (i). 
The main observation in this section is the following:
\begin{proposition} \label{prop: version of (ii)}
	Suppose that the Conditions \ref{cond: main1} \textup{(i) -- (ii)} and one of the Conditions~\ref{cond: compact assumption} and \ref{cond: iii alternative} hold. Let \((\nu^n)_{n = 0}^\infty \subset \W^p (H)\) be a bounded (in the sense of Lemma~\ref{lem: moment estimates}) sequence such that \(\nu^n \to \nu^0\) in \(\W^q (H)\).
	Every sequence \((Q^n)_{n = 1}^\infty\) with \(Q^n \in \cR^n (\nu^n)\) is relatively compact in \(\W^\o (\W^\o (\Theta))\) and any of its \(\o\)-Wasserstein accumulation points is in \(\cR^0 (\nu^0)\). 
\end{proposition} 

Before we prove this proposition, let us deduce Theorem \ref{theo: main1} (ii).

\begin{proof}[Proof of Theorem \ref{theo: main1} \textup{(ii)}]
	Let \(\pi_1 \equiv \pi\) and \(\Pi\) be as in the proof of Theorem \ref{theo: main1} (i) and recall that \(\cA^0 (\nu^0) = \pi (\cC^0 (\nu^0))\) and \(\mathcal{U}^n (\nu^n) = \Pi (\cR^n (\nu^n))\). Furthermore, using that \(\cC^0 (\nu^0) \subset \pi^{-1} (\pi (\cC^0 (\nu^0))) = \pi^{-1} (\cA^0 (\nu^0))\), we also observe that 
	\begin{align*}
		\Pi (\cR^0 (\nu^0)) &= \{ \Pi (Q) \colon Q (\cC^0 (\nu^0)) = 1 \} 
		\\&\subset \{ \Pi (Q) \colon Q \circ \pi^{-1} (\cA^0 (\nu^0)) = 1\} 
		\\&= \{ \Pi (Q) \colon \Pi (Q) (\cA^0 (\nu^0)) = 1\} 
		\\&= \mathcal{U}^0 (\nu^0).
	\end{align*}	
As \(\bigcup_{n = 1}^\infty \cR^n (\nu^n)\) is relatively compact in \(\W^\o (\W^\o (\Theta))\) by Lemma \ref{lem: Rk rel comp}, the set
\[
\bigcup_{n = 1}^\infty \ocA^n (\nu^n) = \Pi \Big( \bigcup_{n = 1}^\infty \cR^n (\nu^n) \Big)
\]
is relatively compact in \(\W^\o (\W^\o (\Omega))\) thanks to the continuity of \(\Pi\).
Hence, the sequence \((Q^n)_{n = 1}^\infty\) from Theorem \ref{theo: main1} (ii) is relatively compact in \(\W^\o (\W^\o (\Omega))\). Let \(P^n \in \cR^n (\nu^n)\) be such that \(\Pi (P^n) = Q^n\). By Proposition \ref{prop: version of (ii)}, any subsequence of \((P^n)_{n = 1}^\infty\) has a further subsequence \((P^{N_n})_{n = 1}^\infty\) that converges in \(\W^\o (\W^\o (\Theta))\) to a measure \(P^0\in\cR^0 (\nu^0)\). Now, \(Q^{N_n} = \Pi (P^{N_n})\) converges in \(\W^\o (\W^\o (\Omega))\) to the measure \(\Pi (P^0) \in \Pi (\cR^0(\nu^0)) \subset \ocA^0 (\nu^0)\). 
The proof of Theorem~\ref{theo: main1}~(ii) is complete.
\end{proof}

It is left to prove Proposition~\ref{prop: version of (ii)}.

\begin{proof}[Proof of Proposition~\ref{prop: version of (ii)}]
	Take a sequence \((\nu^n)_{n = 0}^\infty \subset \W^p (H)\) that is bounded (in the sense of Lemma~\ref{lem: moment estimates}) such that \(\nu^n \to \nu^0\) in \(\W^q (H)\) and let \((Q^n)_{n = 1}^\infty\) be such that \(Q^n \in \cR^n (\nu^n)\). By Lemma \ref{lem: Rk rel comp}, the set \[\bigcup_{n = 1}^\infty \cR^n (\nu^n) \subset \bigcup_{n, m = 1}^\infty \cR^n (\nu^m) = \cR ( \{\nu^m \colon m \in \mathbb{N}\} )\] is relatively compact in \(\W^\o (\W^\o (\Theta))\). Consequently, the sequence \((Q^n)_{n = 1}^\infty\) is relatively compact in \(\W^\o (\W^\o (\Theta))\). It remains to show that every of its \(q\)-Wasserstein accumulation point is in \(\cR^0 (\nu^0)\). To keep our notation simple, we assume that \(Q^n \to Q^0\) in \(\W^\o (\W^\o (\Theta))\). We 
	now use Lemma \ref{lem: mg chara} to
	prove that \(Q^0 \in \cR^0 (\nu^0)\). 
	
Using the i.i.d. assumption of the initial values in the definition of \(\cR^n (\nu^n)\) and the assumption \(\nu^n \to \nu^0\) in \(\W^q (H)\), hence also in \(\W (H)\), it follows as in the proof of \cite[Proposition~2.2~i)]{SnzPoC} that 
\begin{align*}
	Q^0 ( \{ Q \in \W(\Theta) \colon Q \circ X^{-1}_0 = \nu^0 \} ) = 1, 
\end{align*}
i.e., almost all realizations of \(Q^0\) satisfy part (iii.a) from Lemma \ref{lem: mg chara} with initial distribution~\(\nu^0\).
	
	We now turn to the proof of part (iii.b) from Lemma~\ref{lem: mg chara}.
	By Fatou's lemma for weak convergence (see \cite[Theorem~2.4]{feinberg20} for a suitable version), we obtain 
	\begin{equation} \label{eq: Q0 in Wp}
		\begin{split}
			\int \|\mu^X\|^p_p \, Q^0 (d \mu) &\leq \int \liminf_{\nu \to \mu} \|\nu^X\|^p_p \, Q^0 (d \mu) 
			\\&\leq \liminf_{n \to \infty} \int \|\mu^X\|^p_p \, Q^n (d \mu) 
			\\&\leq \sup \Big\{ \frac{1}{n} \sum_{k = 1}^n E^P \big[ \|X^k\|^p_T \big] \colon P \in \cC^n (\nu^n), n \in \mathbb{N} \Big\}.
		\end{split}
	\end{equation}
	As the last term is finite by Lemma \ref{lem: moment estimates}, it follows that \(Q^0 \in \fP (\W^p (\Theta))\). In particular, almost all realizations of \(Q^0\) have the property (iii.b) from Lemma \ref{lem: mg chara}.
	
	Finally, we prove (iii.c) from Lemma \ref{lem: mg chara}.
	Take \(y^* \in \mathcal{D} (A^*), g \in \mathcal{C}^2_c, s, t \in \mathbb{Q}_+ \cap [0, T], s < t\) and \(\mathfrak{t} \in \mathcal{T}_s\). Recall \eqref{eq: Mr}, i.e., that, for \((r, \omega, m, \mu) \in [0, T] \times \Theta \times \W^\o (\Theta)\),
	\[
	\oM_r (\omega, m, \mu) = g (\langle \omega (r), y^*\rangle_H) - \int_0^r \int \mathcal{L}_{g, y^*} (f, u, \omega, \mu^X_u) \, m (du, df).
	\]
	For \(\mu \in \W^\o (\Theta)\), we define 
	\[
	\oZ^k (\mu) := \int \big[ k \wedge ( \oM_t (\omega, m, \mu) - \oM_s (\omega, m, \mu) ) \vee (- k) \big]  \mathfrak{t} (\omega, m) \,  \mu (d \omega, dm),
	\]
	and 
	\[
	\oZ (\mu) := \liminf_{k \to \infty} \oZ^k (\mu).
	\]
	By Lemma~\ref{lem: Mr cont}, \(\mathsf{M}_r \colon \Theta \times \fP^q (\Theta) \to \bR\) is continuous for every \(r \in [0, T]\). 
	Hence, thanks to \cite[Theorem~8.10.61]{bogachev}, the map \(\oZ^k \colon \fP^q (\Theta) \to \bR\) is continuous and consequently, \(\oZ\) is Borel measurable. 
	Thanks to Condition \ref{cond: main1} (ii), we have 
	\begin{align*}
	| \oM_t (\omega, m, \mu) - \oM_s (\omega, m, \mu) | &\leq C \big( 1 + \|\omega\|^2_T + \|\mu^X\|^2_p \big).
	\end{align*}
Hence, since 
\[
| k \wedge ( \oM_t (\omega, m, \mu) - \oM_s (\omega, m, \mu) ) \vee (- k) | \leq | \oM_t (\omega, m, \mu) - \oM_s (\omega, m, \mu) |,
\]
the dominated convergence theorem yields that 
\[
\mu \in \W^p (\Theta) \ \Longrightarrow \ \oZ (\mu) = \int (\oM_t (\omega, m, \mu) - \oM_s (\omega, m, \mu) ) \mathfrak{t} (\omega, m) \, \mu (d \omega, dm).
\]

We now prove that \(Q^0\)-a.s. \(\oZ = 0\). By Lemma \ref{lem: mg chara}, as \(\mathcal{D} (A^*), \mathcal{C}^2_c\) and \(\mathcal{T}_s\) are countable, this 
implies that almost all realizations of \(Q^0\) satisfy (iii.c) from Lemma \ref{lem: mg chara}. In summary, we then can conclude that \(Q^0 (\cC^0 (\nu^0)) = 1\), which means that \(Q^0 \in \cR^0 (\nu^0)\). 

The proof of \(Q^0\)-a.s. \(\oZ= 0\) uses a strategy we learned from \cite{bhatt1998interacting}, cf. also \cite{C23,C23c}. It is divided into two steps. First, we prove that 
\begin{align} \label{eq: first main oZ = 0}
	\lim_{n \to \infty} E^{Q^n} \big[ | \oZ | \big] = E^{Q^0} \big[ | \oZ | \big], 
\end{align} 
and afterwards, we show that 
\begin{align} \label{eq: second main oZ = 0}
	\lim_{n \to \infty} E^{Q^n} \big[ | \oZ |^2 \big] = 0.
\end{align} 
Obviously, \eqref{eq: first main oZ = 0} and \eqref{eq: second main oZ = 0} yield that \(E^{Q^0} [ | \oZ | ] = 0\), which proves \(Q^0\)-a.s. \(\oZ = 0\). 

We proceed with the proofs for \eqref{eq: first main oZ = 0} and \eqref{eq: second main oZ = 0}. By the triangle inequality, we observe that 
\begin{equation*} 
	\begin{split}
	| E^{Q^n} [ | \oZ | ] - E^{Q^0} [ | \oZ | ] | &\leq | E^{Q^n} [ | \oZ | ] - E^{Q^n} [ | \oZ^k | ] |  \\&\hspace{1cm}+ 	| E^{Q^n} [ | \oZ^k | ] - E^{Q^0} [ | \oZ^k | ] | 
	\\&\hspace{1cm}+ 	| E^{Q^0} [ | \oZ^k | ] - E^{Q^0} [ | \oZ | ] |
 \\&=: I_{n, k} + II_{n, k} + III_k.
	\end{split}
\end{equation*}
First, notice that \(II_{n, k} \to 0\) as \(n \to \infty\) for every \(k > 0\), as \(\oZ^k\) is bounded and continuous on \(\W^\o (\Theta)\). We now discuss \(I_{n, k}\) and \(III_k\). By definition of \(\cR^n (\nu^n)\), there is a measure \(P^n \in \cC^n (\nu^n)\) such that \(Q^n = P^n \circ \Y_n^{-1}\). 
We obtain 
\begin{align*}
 I_{n, k} &\leq \frac{C}{n} \sum_{j = 1}^n E^{P^n} \big[ | (\oM_t - \oM_s) - k \wedge ( \oM_t - \oM_s ) \vee (- k) | \circ (X^j, M^j, \X_n (X))\big]
	\\&\leq \frac{C}{n} \sum_{j = 1}^n E^{P^n} \big[ | \oM_t - \oM_s | \1_{\{| \oM_t - \oM_s|\, >\, k\}} \circ (X^j, M^j, \X_n (X))\big]
	\\&\leq \frac{C}{k^{p/2 - 1}} \frac{1}{n} \sum_{j= 1}^n E^{P^n} \big[ | \oM_t - \oM_s |^{p /2} \circ (X^j, M^j, \X_n (X))\big]
	\\&\leq \frac{C}{k^{p/2 - 1}} \frac{1}{n} \sum_{j = 1}^n E^{P^n} \big[ 1 + \|X^j\|^{p}_T + \|\X_n (X)\|^p_p \big]
	\\&= \frac{C}{k^{p/2 - 1}} \Big( 1 + \frac{1}{n} \sum_{j = 1}^n E^{P^n} \big[ \|X^j\|^{p}_T \big] \Big)
	\\&\leq \frac{C}{k^{p/2 - 1}},
\end{align*}
where the constant is independent of \(n\) by the moment estimate from Lemma \ref{lem: moment estimates}. Similarly, we obtain that 
\begin{align*}
	III_k &\leq \frac{C}{k^{p/2 - 1}} \iint | \oM_t (\theta, \mu) - \oM_s (\theta, \mu) |^{p/2} \, \mu (d \theta) \, Q^0 (d \mu) 
	\\&\leq \frac{C}{k^{p/2 - 1}} \Big( 1 +\int \|\mu^X\|^p_p \, Q^0 (d \mu) \Big).
\end{align*}
The last term is finite by \eqref{eq: Q0 in Wp}. In summary, \(I_{n, k} + III_k \to 0\) as \(k \to \infty\) uniformly in~\(n\). Together with our previous observation that \(II_{n, k} \to 0\) as \(n \to \infty\) for fixed \(k > 0\), we conclude that \eqref{eq: first main oZ = 0} holds. 

Finally, it remains to prove \eqref{eq: second main oZ = 0}. Notice that 
\begin{align*}
	E^{Q^n} \big[ \oZ^2 \big] = \frac{1}{n^2} \sum_{i, j = 1}^n E^{P^n} \big[ \oZ( \delta_{(X^i, M^i)} ) \oZ (\delta_{(X^j, M^j)}) \big].
\end{align*}
	Take \(1 \leq i < j \leq n\). By It\^o's formula, \(P^n\)-a.s.
\begin{align*}
\oK^i &:= \oM (X^i, M^i, \X_n (X)) - g (\langle X^i_0, y^* \rangle_H) 
\\&= \int_0^\cdot g' (\langle X^i_u, y^* \rangle_H) \langle \sigma^* (\xi^i_u, u, X^i, \X_n (X_{\cdot \wedge u})) y^*, d W^i_u \rangle_U, 
\end{align*}
where \(\xi^i\) and \(W^i\) are as in the definition of \(P^n\). By the independence of \(W^i\) and \(W^j\), we obtain that the quadratic variation of \(\oK^i\) and \(\oK^j\) vanishes. As \(\oK^i\) and \(\oK^j\) are square integrable \(P^n\)-\(\mathbf{O}\)-martingales (see Lemma \ref{lem: mg chara N}), this means that the product \(\oK^i \oK^j\) is a \(P^n\)-\(\mathbf{O}\)-martingale. Consequently, using that \(\oK^i, \oK^j\) and \(\oK^i \oK^j\) are \(P^n\)-\(\mathbf{O}\)-martingales, we obtain 
\begin{align*}
	E^{P^n}  \big[ \oZ(\delta_{(X^i, M^i)}) & \oZ(\delta_{(X^j, M^j)})\big] 
	\\&= E^{P^n} \big[ \big( \oK^i_t \oK^j_t - \oK^i_t \oK_s^j - \oK^i_s \oK^j_t + \oK^i_s \oK^j_s \big) \z (X^i, M^i) \z (X^j, M^j) \big]
	\\&= E^{P^n} \big[ \big( \oK^i_s \oK^j_s - \oK^i_s \oK_s^j - \oK^i_s \oK^j_s + \oK^i_s \oK^j_s \big) \z (X^i, M^i) \z (X^j, M^j) \big] 
	\\&= 0.
\end{align*}
This implies that
\[
\frac{1}{n^2} \sum_{i, j = 1}^n E^{P^n} \big[ \oZ(\delta_{(X^i, M^i)}) \oZ(\delta_{(X^j, M^j)})\big] =  \frac{1}{n^2} \sum_{k = 1}^n E^{P^n} \big[ \oZ(\delta_{(X^k, M^k)} \big].                                        \]
Using that 
\begin{align*}
E^{P^n} \Big[ \sup_{r \in [0, T]} |\oK^k_r|^2 \Big] &\leq E^{P^n} \Big[ \|g'\|_\infty \int_0^T \|\sigma^* (\xi^k_u, u, X^k, \X_n (X_{\cdot \wedge u})) y^* \|_U^2 \, du \Big]
 \\&\leq C \Big( 1 + E^{P^n} \big[ \|X^k\|_T^2 \big] + E^{P^n} \big[ \| \X_n (X) \|^2_p \big] \Big)
 \\&\leq C \Big( 1 + E^{P^n} \big[ \|X^k\|_T^p \big] + \frac{1}{n} \sum_{i = 1}^n E^{P^n} \big[ \|X^i\|_T^p \big]  \Big), 
\end{align*}
which follows from the linear growth assumption given by Condition \ref{cond: main1} (ii) and Burkholder's inequality, we conclude from Lemma \ref{lem: moment estimates} that 
\begin{align*}
	\frac{1}{n} \sum_{k = 1}^n E^{P^n} \big[ \oZ(\delta_{(X^k, M^k)}^2 \big] &\leq C.
\end{align*}
In summary, we have 
\begin{align*}
	E^{Q^n} \big[ \oZ^2 \big] \leq \frac{C}{n}, 
\end{align*}
which proves \eqref{eq: second main oZ = 0}. The proof of Proposition~\ref{prop: version of (ii)} is complete.
\end{proof}

	\subsection{Proof of Theorem \ref{theo: main1} (iii)}
 Given Theorem~\ref{theo: main1}~(i) and (ii), the proof for \cite[Theorem~2.5~(iii)]{C23c} yields the claim. For reader's convenience, we reproduce the argument here. 

 \smallskip 
	We use the notation from Theorem \ref{theo: main1} (iii). 
	Using the compactness of \(\ocA^n (\nu^n)\), which is due to Theorem \ref{theo: main1} (i), and standard properties of the limes superior, there exists a subsequence \((N_n)_{n = 1}^\infty\) of \(1, 2, \dots\) and measure \(Q^{N_n} \in \ocA^{N_n} (\nu^{N^n})\) such that
	\[
	\limsup_{n \to \infty} \sup_{Q \in \mathcal{U}^n (\nu^n)} E^Q \big[ \psi \big] = \lim_{n \to \infty} E^{Q^{N_n}} \big[ \psi \big].
	\]
	By Theorem \ref{theo: main1} (ii), there is a subsequence of \((Q^{N_n})_{n = 1}^\infty\) that converges in \(\W^\o (\W^\o (\Omega))\) to a measure \(Q^0 \in \mathcal{U}^0 (\nu^0)\). Hence, by the properties (upper semicontinuity and growth) of \(\psi\) and \cite[Lemma~4.11]{C23c}, we get
	\[
	\lim_{n \to \infty} E^{Q^{N_n}} \big[ \psi \big] \leq E^{Q^0} \big[ \psi \big] \leq \sup_{Q \in \mathcal{U}^0 (\nu^0)} E^Q \big[ \psi \big].
	\]
	This completes the proof. 
	\qed

	\subsection{Proof of Theorem \ref{theo: main1} (iv)}
	The strategy of proof is inspired by the proof for \cite[Theorem~2.12]{LakSIAM17}, cf. also the proof of \cite[Theorem~2.5~(iv)]{C23c}. In particular, we learned the idea to use the Krein--Milman theorem from the proof of \cite[Theorem~2.12]{LakSIAM17}.
	Let us start with an auxiliary result whose proof is postponed to the end of this section. 
	\begin{lemma} \label{lem: approx point mass}
		Assume that the Conditions~\ref{cond: main1} and \ref{cond: main2} hold. Let \((\nu^n)_{n = 0}^\infty \subset \W^p (H)\) be a sequence such that \(\nu^n \to \nu^0\) in \(\W^p (H)\) and take \(P \in \cA^0 (\nu^0)\). Then, there exists a sequence \((Q^{n})_{n = 1}^\infty\) with \(Q^{n} \in \mathcal{U}^{n} (\nu^{n})\) such that \(Q^{n} \to \delta_{P}\) in \(\W^p (\W^p (\Omega))\).
	\end{lemma}

With Lemma~\ref{lem: approx point mass} at hand, we are ready to prove Theorem \ref{theo: main1} (iv). 
Set 
\[
L := \Big\{ Q \in \fP^q (\fP^q (\Omega)) \colon \, \exists \, Q_n \in \mathcal{U}^n (\nu^n) \text{ such that } Q = \lim_{n \to \infty} Q^n \text{ in } \fP^q (\fP^q (\Omega)) \Big\}. 
\] 
In the following we show that \(\mathcal{U}^0 (\nu^0) \subset L\), which proves the claim.
As each \(\ocA^n (\nu^n)\) is convex by the convexity of \(\cA^n (\nu^n)\) from Corollary~\ref{coro: conv}, \(L\) is also convex. Furthermore, \(L\) is closed (in \(\W^\o (\W^\o (\Omega))\)) by \cite[Proposition~1.1.47~(a); or Proposition~7.1.20]{HP_97}. Notice that \(\mathcal{U}^0 (\nu^0) = \W (\cA^0 (\nu^0))\) is convex and, by \cite[Theorem~15.9]{charalambos2013infinite}, its extreme points are given by \(\{\delta_{P^0} \colon P^0 \in \cA^0 (\nu^0)\}\). Recall from the proof of Theorem \ref{theo: main1} (i) that \(\cA^0 (\nu^0)\) is non-empty and compact in \(\fP^q (\Omega)\). Thus, \(\ocA^0 (\nu^0)\) is nonempty and compact in \(\W^\o (\W^\o (\Omega))\) and both \(\fP^q (\fP^q (\Omega))\) and \(\fP (\fP^q (\Omega))\) induce the same topology on \(\ocA^0 (\nu^0)\). 
Thanks to the Krein--Milman theorem (\cite[Theorem~7.68]{charalambos2013infinite}), we have
\[
\mathcal{U}^0 (\nu^0) = \overline{\on{co}}\, \big[ \{\delta_{P^0} \colon P^0 \in \cA^0 (\nu^0)\}\big],
\]
where \(\overline{\on{co}}\) denotes the closure (in \(\W^\o (\W^\o (\Omega))\)) of the convex hull. Consequently, as \(L\) is convex and closed, \(\ocA^0 (\nu^0) \subset L\)
follows from Lemma~\ref{lem: approx point mass} that implies
\(
\{\delta_{P^0} \colon P^0 \in \cA^0 (\nu^0)\} \subset L.
\) 
The proof of Theorem~\ref{theo: main1}~(iv) is complete.
\qed
\vspace{0.25cm}

It is left to prove Lemma \ref{lem: approx point mass}. We prepare the proof with a version of Skorokhod's coupling theorem. 

\begin{lemma} \label{lem: skorokhod coupling} 
	Let \(E\) be a Polish space and let \((\mu^n)_{n = 0}^\infty \subset \W (E)\) be a sequence such that \(\mu^n \to \mu^0\) in \(\W (E)\). Further, take a probability space \((\Omega', \cF', P')\) that supports an \(E\)-valued random variable \(Z^0\) with distribution \(\mu^0\). Then, on the standard extension 
	\[
	\Omega' \times [0, 1], \ \cF\otimes \mathcal{B} ([0, 1]), \ P' \otimes \llambda, \qquad \llambda = \text{Lebesgue measure}, 
	\] 
	there exist random variables \(Z^1, Z^2, \dots\) such that \(Z^n \sim \mu^n\) and \(Z^n \to Z^0\) almost surely.
\end{lemma} 
\begin{proof}
	By \cite[Theorem~8.5.4]{bogachev}, there are Borel maps \(F_n \colon [0, 1] \to E\) such that \(\llambda \circ F_n^{-1} = \mu^n\) and \(\llambda\)-a.s. \(F_n \to F_0\). Furthermore, by \cite[Corollary~8.18]{Kallenberg}, on the standard extension, there exists a random variable \(U\), uniformly distributed on \([0, 1]\), such that a.s. \(Z^0 = F_0 (U)\). Now, \(Z^n := F_n (U) \sim \mu^n\) and a.s. \(Z^n \to Z^0\). This is the claim.
\end{proof}

\begin{proof}[Proof of Lemma \ref{lem: approx point mass}]
	We tailor a coupling idea as outlined in \cite{LakLN,SM} to our setting.
	Let \((\nu^n)_{n = 0}^\infty \subset \W^p (H)\) be such that \(\nu^n \to \nu^0\) in \(\W^p (H)\) and take \(P \in \cA^0 (\nu^0)\). By definition, possibly on a standard extension of the stochastic basis \((\Omega, \cF, \mathbf{F}, P)\) whose notation we ignore for simplicity, there exists a standard cylindrical Brownian motion \(W\) such that \(P\)-a.s., for all \(t \in [0, T]\),
	\[
	X_t = S_t X_0 + \int_0^t S_{t - s} b (\f_s, s, X, P^X_s) \, ds + \int_0^t S_{t - s} \sigma (\f_s, s, X, P^X_s) \, d W_s.
	\]
	Using the usual product construction, we may construct a filtered probability space (whose expectation we denote by \(E\)) that supports a sequence \((X^n, W^n)_{n = 1}^\infty\) of independent copies of \((X, W)\). In particular, for \(t \in [0, T]\),
	\[
	X^k_t = S_t X^k_0 + \int_0^t S_{t - s} b (\f_s (X^k), s, X^k, P^X_s) \, ds + \int_0^t S_{t - s} \sigma (\f_s (X^k), s, X^k, P^X_s) \, d W^k_s.
	\]
	Furthermore, by Lemma~\ref{lem: skorokhod coupling}, possibly passing again to a standard extension that we still denote by \((\Omega, \cF, \mathbf{F}, P)\), we may assume that there exist \(\cF_0\)-measurable random variables \(\{X^{n, k}_0 \colon n \in \mathbb{N}, \, k \leq n\}\) such that \(X^{n, 1}_0, \dots, X^{n, n}_0\) are i.i.d. with distribution \(\nu^n\) and almost surely \(X^{n, k}_0 \to X^k_0\) as \(n \to \infty\). In particular, as \(\nu^n \to \nu^0\) in \(\W^p (H)\), we also have \(X^{n, k}_0 \to X^k_0\) in \(L^p\), see \cite[Proposition~A.1]{LakSPA15} and \cite[Lemma~5.10]{Kallenberg}.
	Thanks to the Conditions \ref{cond: main1} (ii) and \ref{cond: main2}, a standard contraction argument (see, e.g., the proof of \cite[Theorem A.1]{C23}) shows that (on our underlying filtered probability space) there are continuous \(H^n\)-valued processes \(Y^n = (Y^{n, 1}, \dots, Y^{n, n})\) with dynamics
	    \begin{align*}
	Y^{n, k}_t = S_t X^{n, k}_0 &+ \int_0^t S_{t - s} b (\f_s (X^k), s, Y^{n, k}, \X_n (Y^n_{\cdot \wedge s})) \, ds 
	\\&+ \int_0^t S_{t - s} \sigma (\f_s (X^k), s, Y^{n, k}, \X_n (Y^n_{\cdot \wedge s})) \, d W^k_s, \quad t \in [0, T].
	\end{align*}
Using the inequality from \cite[Lemma 4.2]{C23}, and our Lipschitz assumptions, for every \(t \in [0, T]\), we obtain that 
\begin{align*}
	E \Big[ & \sup_{s \in [0, t]} \|Y^{n, k}_s - X^k_s\|^p_H \Big] 
	\\&\leq C\Big( E \Big[ \|X^{n, k}_0 - X^k_0\|_H^p \Big] + \int_0^t E\Big[ \|Y^{n, k}_s - X^k_s\|^p_H + \w_p (\X_n (Y^n_{\cdot \wedge s}), P^X_s)^p \Big] \, ds \Big)
	\\&\leq C\Big( E \Big[ \|X^{n, k}_0 - X^k_0\|_H^p \Big] + \int_0^t E\Big[\sup_{r \in [0, s]} \|Y^{n, k}_r - X^k_r\|^p_H + \w_p (\X_n (Y^n_{\cdot \wedge s}), P^X_s)^p \Big] \, ds \Big).
\end{align*}
Gronwall's lemma yields that 
\begin{align} \label{eq: Gron 1}
	E \Big[ \sup_{s \in [0, t]} \|Y^{n, k}_s - X^k_s\|^p_H \Big] \leq C \Big( E \Big[ \|X^{n, k}_0 - X^k_0\|_H^p \Big] + \int_0^t E \Big[\w_p (\X_n (Y^n_{\cdot \wedge s}), P^X_s)^p \Big] ds \Big).
\end{align}
We set \(Z^n := (X^1, \dots, X^n)\). Using the coupling \(\frac{1}{n} \sum_{k = 1}^n \delta_{(Y^{n, k}, X^k)}\), we observe that 
\begin{align} \label{eq: coupl 1}
	\w_p(\X_n (Y^n_{\cdot \wedge t}), \X_n (Z^n_{\cdot \wedge t}))^p \leq \frac{1}{n} \sum_{k= 1}^n \sup_{s \in [0, t]} \|Y^{n, k}_s - X^k_s\|^p_H.
\end{align}
Hence, using the triangle inequality, \eqref{eq: Gron 1} and \eqref{eq: coupl 1}, we obtain that 
\begin{align*}
	E \Big[ & \w_p (\X_n (Y^n_{\cdot \wedge t}), P^X_t)^p \Big] 
	\\&\leq C \Big( E \Big[ \|X^{n, k}_0 - X^k_0\|_H^p \Big] + \int_0^t E \Big[\w_p (\X_n (Y^n_{\cdot \wedge s}), P^X_s)^p\Big] \, ds \Big) + E \Big[ \w_p (\X_n (Z^n_{\cdot \wedge t}), P^X_t)^p \Big].
\end{align*}
Using Gronwall's lemma once again (notice that \(t \mapsto E [ \w_p (\X_n (Z^n_{\cdot \wedge t}), P^X_t)^p]\) is increasing), we get that 
\begin{align} \label{eq: main1}
	E \Big[ \w_p (\X_n (Y^n), P)^p \Big] \leq C \Big( E \Big[ \|X^{n, k}_0 - X^k_0\|_H^p \Big] + E \Big[ \w_p (\X_n (Z^n), P)^p \Big] \Big).
\end{align}
As \(X^1, X^2, \dots\) are i.i.d. copies of \(X\), it follows from \cite[Corollary 2.14]{LakLN} that 
\[
E \Big[ \w_p (\X_n (Z^n), P)^p \Big] \to 0 \text{ as } n \to \infty.
\]
Using further that \(X^{n, k}_0 \to X^k_0\) in \(L^p\) by construction, we conclude from \eqref{eq: main1} that 
\[
E \Big[ \w_p (\X_n (Y^n), P)^p \Big] \to 0 \text{ as } n \to \infty.
\]
Let \(Q^n\) be the law of \(\X_n (Y^n)\). Then, as 
\[
\widehat{\w}_p (Q^n, \delta_P)^p = E \Big[ \w_p (\X_n (Y^n), P)^p \Big], 
\]
it follows that \(Q^n \to \delta_P\) in \(\W^p (\W^p (\Omega))\). 
To complete the proof, we need to explain that \(Q^n \in \mathcal{U}^{n} (\nu^{n})\), which is not immediate as the controls in the definition of \(Y^1, \dots, Y^n\) are not of feedback type for this sequence. It turns out to be useful to relate \(Q^n\) to relaxed control rules. Notice that \(\overline{Q}^n := P \circ (Y^1, \dots, Y^n, \delta_{\f_t (Y^1)} \, dt, \dots, \delta_{\f_t (Y^n)} \, dt )^{-1}\in \cC^n (\nu^n)\). Hence, Lemma~\ref{lem: nonlinear SPDE rel RCR} yields that \(\overline{Q}^n \circ (X^1, \dots, X^n)^{-1} \in \cA^n (\nu^n)\) and consequently, 
\[
Q^n = \overline{Q}^n \circ (X^1, \dots, X^n)^{-1} \circ \X_n^{-1} \in \mathcal{U}^{n} (\nu^{n}).
\] 
The lemma is proved.
\end{proof}

		\subsection{Proof of Theorem \ref{theo: main1} (v)}
   Given Theorem~\ref{theo: main1}~(iv), the proof for \cite[Theorem~2.5~(v)]{C23c} can be used without modification. For reader's convenience, we reproduce the argument here, using the notation from Theorem \ref{theo: main1}~(v).

 \smallskip  
Take an arbitrary measure \(Q^0 \in \ocA^0 (\nu^0)\). Then, by Theorem \ref{theo: main1} (iv), there exists a sequence \(Q^{n} \in \ocA^{n} (\nu^{n})\) such that \(Q^{n} \to Q^0\) in \(\W^\o (\W^\o(\Omega))\). Now, by \cite[Lemma~4.11]{C23c}, and our assumptions on \(\psi\), 
\begin{align*}
	E^{Q^0} \big[ \psi \big] \leq \liminf_{n \to \infty} E^{Q^{n}} \big[ \psi \big] 
	\leq \liminf_{n \to \infty} \sup_{Q\in \ocA^{n} (\nu^{n})} E^{Q} \big[ \psi \big].
\end{align*}
As \(Q^0\) was arbitrary, we get 
\[
\sup_{Q \in \ocA^0 (\nu^0)} E^Q \big[ \psi \big] \leq \liminf_{n \to \infty} \sup_{Q \in \mathcal{U}^n (\nu^n)} E^Q \big[ \psi \big].
\]
The proof is complete.
\qed
	
		\subsection{Proof of Theorem \ref{theo: main1} (vi)}
     Given Theorem~\ref{theo: main1}~(iii) and (v), the proof for \cite[Theorem~2.5~(vi)]{C23c} can be used without modification. For reader's convenience, we reproduce the argument here.

 \smallskip 
		Let \(\psi \colon \W^\o (\Omega) \to \bR\) be a continuous function with the property \eqref{eq: bound property}. Then, by Theorem~\ref{theo: main1} (iii) and (v), for every sequence \((\nu^n)_{n = 0}^\infty \subset \W^p (H)\) with \(\nu^n \to \nu^0\) in \(\W^p (H)\), we get
		\[
		\sup_{Q \in \ocA^n (\nu^n)} E^Q \big[ \psi \big] \to \sup_{Q \in \ocA^0 (\nu^0)} E^Q \big[ \psi \big], \quad n \to \infty.
		\]
		Now, it follows from \cite[Theorem on pp. 98--99]{remmert} that \(\nu \mapsto \sup_{Q \in \ocA^0 (\nu)} E^Q [ \psi ]\) is continuous from \(\W^p (H)\) into \(\bR\) and that the convergence \eqref{eq: compact convergence} holds.
		The proof is complete. \qed 
  
		\subsection{Proof of Theorem \ref{theo: main1} (vii)}
Given Theorem~\ref{theo: main1}~(i), (ii) and (iv), the claim can be proved identically to \cite[Theorem~2.5~(vii)]{C23c}. For reader's convenience, we reproduce the argument here.

 \smallskip 
Recall that \(\widehat{\w}_\o\) denotes the \(\o\)-Wasserstein metric on \(\W^\o (\W^\o (\Omega))\). By virtue of \cite[Theorem on pp. 98--99]{remmert}, it suffices to prove that for every sequence \( (\nu^n)_{n =0}^{\infty}\subset \W^p (H)\) with \(\nu^n \to \nu^0\) in \(\W^p (H)\),
		\begin{align*}
		\mathsf{h}	\, ( \ocA^n (\nu^n), \ocA^0(\nu^0)) = \max \Big\{ \max_{Q \in \ocA^n (\nu^n)} \widehat{\w}_\o (Q, \ocA^0 (\nu^0)) , \max_{Q \in \ocA^0 (\nu^0)} \widehat{\w}_\o (Q, \ocA^n (\nu^n))\Big\} \to 0
		\end{align*}
	as \(n \to \infty\). Notice that the maxima are attained by the compactness of the sets \(\ocA^n (\nu^n)\) and \(\ocA^0 (\nu^0)\) that follow from Theorem \ref{theo: main1} (i).
	
	We start investigating the first term. 	
	By the compactness of each \(\ocA^n (\nu^n)\), for every \(n  \in \mathbb{N}\), there exists a measure \(Q^n \in \ocA^n (\nu^n)\) such that 
	\[
	\max_{Q \in \ocA^n (\nu^n)} \widehat{\w}_\o (Q, \ocA^0 (\nu^0)) = \widehat{\w}_\o (Q^n, \ocA^0 (\nu^0)).
	\]
	By Theorem \ref{theo: main1} (ii), every subsequence of \(1, 2, \dots\) has a further subsequence \((N_n)_{n = 1}^\infty\) such that \((Q^{N_n})_{n = 1}^\infty\) converges in \(\W^\o (\W^\o (\Omega))\) to a measure \(Q^0 \in \ocA^0 (\nu^0)\). Now, by the continuity of the distance function, we have 
	\[
	\widehat{\w}_\o (Q^{N_n}, \ocA^0 (\nu^0)) \to \widehat{\w}_\o (Q^0, \ocA^0 (\nu^0)) = 0.
	\]
	We conclude that
	\[
		\max_{Q \in \ocA^n (\nu^n)} \widehat{\w}_\o (Q, \ocA^0 (\nu^0)) = \widehat{\w}_\o (Q^n, \ocA^0 (\nu^0)) \to 0 \text{ as } n \to \infty.
	\]
	
	We turn to the second term. By the compactness of \(\ocA^0 (\nu^0)\), for every \(n \in \mathbb{N}\), there exists a measure \(R^n \in \ocA^0 (\nu^0)\) such that 
	\[
	\max_{Q \in \ocA^0 (\nu^0)} \widehat{\w}_\o (Q, \ocA^n (\nu^n)) = \widehat{\w}_\o (R^n, \ocA^n (\nu^n)).
	\]
	Let \((N^n_1)_{n = 1}^\infty\) be an arbitrary subsequence of \(1, 2, \dots\). Again by compactness of \(\ocA^0 (\nu^0)\), there exists a subsequence \((N^n_2)_{n = 1}^\infty \subset (N^n_1)_{n = 1}^\infty\) such that \((R^{N^n_2})_{n = 1}^\infty\) converges in \(\W^\o (\W^\o (\Omega))\) to a measure \(R^0 \in \ocA^0 (\nu^0)\). 
	By Theorem \ref{theo: main1} (iv), there exists a sequence  \((Q^{N^n_2})_{n = 1}^\infty\) such that \(Q^{N^n_2} \in \ocA^{N^n_2}(\nu^{N^n_2})\) and \(Q^{N^n_2} \to R^0\) in \(\W^\o (\W^\o (\Omega))\). Finally, 
	\[
	 \widehat{\w}_\o (R^{N^n_2}, \ocA^{N^n_2} (\nu^{N^n_2})) \leq \widehat{\w}_\o (R^{N^n_2}, Q^{N^n_2}) \leq  \widehat{\w}_\o (R^{N^n_2}, R^0) + \widehat{\w}_\o(R^0, Q^{N^n_2}) \to 0.
	\]
	As \((N^n_1)_{n = 1}^\infty\) was arbitrary, this proves that 
	\[
	 \widehat{\w}_\o (R^n, \ocA^n (\nu^n)) \to 0.
	\]
	In summary, \(\ocA^n (\nu^n) \to \ocA^0 (\nu^0)\) in the Hausdorff metric topology. \qed

\appendix
\section{An existence theorem for SPDEs}

In this appendix, we provide an existence theorem for semilinear SPDEs with continuous path-dependent coefficients that becomes useful in the proof of our main theorem. The result can be viewed as an extension of some existence results from \cite{gatarekgoldys,GMR}. As its proof follows well-trodden paths, we only sketch it.

\smallskip
Let \(\mu \colon [0, T] \times \Omega \to H\) and \(a \colon [0, T] \times \Omega \to L (U, H)\) be Borel functions that are predictable. Furthermore, as in Section~\ref{sec: main1}, let \(A \colon D (A) \subset H \to H\) be the generator of a strongly continuous semigroup \((S_t)_{t \geq 0}\) on the Hilbert space \(H\), let \(\j \colon [0, T] \to [0, \infty]\) be a Borel function that satisfies \eqref{eq: DZ cond} for \(\alpha \in (0, 1/2)\), \(p > 2\), and \(\rho \in (0, 1 - p/2)\).

\begin{condition} \label{cond: existence1}
 \begin{enumerate}
     \item[\textup{(i)}] \(\mu\) and \(a\) are continuous. 
         \item[\textup{(ii)}] There exists a constant \(C > 0\) such that 
		\begin{align*}
			\| \mu (t, \omega) \|_H + \| a(t, \omega)\|_{L (U, H)} & \leq C \big[ 1 + \|\omega\|_t \big],  
			\\
			\|S_s a(t, \omega) \|_{L_2 (U, H)} &\leq \j (s) \big[ 1 + \|\omega\|_t \big],  
		\end{align*}
		for all \(s, t \in [0, T]\) and \(\omega \in \Omega\).
 \end{enumerate}
\end{condition}

\begin{condition} \label{cond: compact1}
The operator \(A\) generates a compact semigroup, i.e., \(S_t\) is a compact operator for every \(t > 0\).
\end{condition}

\begin{condition} \label{cond: compact2}
There is a Riesz basis $(e_k)_{k=1}^\infty \subset H$ with the following properties:
\begin{enumerate}
	\item[\textup{(i)}] There exists a sequence $(\lambda_k)_{k=1}^\infty \subset \mathbb{R}$ such that $\lambda_k>0$ and 
	\begin{align*} 
		S_t^{*}e_k=e^{-\lambda_k t}e_k\quad \text{for all $k\in\mathbb N$}.
	\end{align*} 
	\item[\textup{(ii)}] There exists a sequence $(\cCk)_{k=1}^\infty\subset \mathbb R_{+}$ such that
	\begin{align*}
		\sum\limits_{k=1}^\infty  \cCk^2 \lambda_k^{-\rho} <\infty,
	\end{align*}
	and
	\begin{align*}
		|\langle \mu (t, \omega),e_k\rangle_H |^2 
		+ \| a^* (t, \omega) e_k\|^2_{U}&\leq \cCk^2\, \big[1+\|\omega\|_t^2\big]
	\end{align*}
	for all $(t, \omega, k)\in [0, T] \times \Omega \times \mathbb N$.
\end{enumerate}
\end{condition} 

\begin{theorem} \label{theo: existence appendix}
	Suppose that Condition~\ref{cond: existence1} holds and in addition assume either Condition~\ref{cond: compact1} or \ref{cond: compact2}. Then, for every \(\nu \in \W(H)\), the SPDE
	\[
	d Y_t = AY_t \, dt + \mu (t, Y) \, dt + a (t, Y) \, d W_t, \quad Y_0 \sim \nu, 
	\] 
	has a martingale solution.\footnote{A martingale solution is a probabilistically weak and analytically mild solution (see \cite[Definition~3.1]{gawa}).} Here, \(W\) is a standard cylindrical Brownian motion. 
\end{theorem}
\begin{proof}
The proof follows the usual path, i.e., approximation of \(\mu\) and \(a\) by Lipschitz coefficients, establishing tightness and then employing a martingale problem argument. 

\smallskip
{\em Step 1: The approximation sequence}. By an inspection of the proof for \cite[Lemma~4]{GMR}, using Condition~\ref{cond: existence1}, it follows that there exist Borel functions \[\mu^n \colon [0, T] \times \Omega \to H, \quad a^n \colon [0, T] \times \Omega \to L (U, H)\] that are predictable and possess the following properties:
\begin{enumerate}
    \item[(a)] There exists a constant \(C > 0\), that does not depend on \(n\), such that 
		\begin{align*}
			\| \mu^n (t, \omega) \|_H + \| a^n (t, \omega)\|_{L (U, H)} & \leq C \big[ 1 + \|\omega\|_t \big], 
			\\
			\|S_s a^n (t, \omega) \|_{L_2 (U, H)} &\leq C \j (s) \big[ 1 + \|\omega\|_t \big], 
		\end{align*}
		for all \(s, t \in [0, T], \omega \in \Omega\) and \(n \in \mathbb{N}\).
    \item[(b)]
    For every \(n \in \mathbb{N}\), there exists a constant \(C = C_n > 0\) such that 
    \begin{align*}
        \| S_s (\mu^n (t, \omega) - \mu^n (t, \alpha)) \|_H & \leq C\, \| \omega - \alpha \|_t, \\
        \| S_s (a^n (t, \omega) - a^n (t, \alpha) ) \|_{L_2 (U, H)} & \leq C \j (s) \, \| \omega - \alpha \|_t,
    \end{align*}
    for all \(s, t \in [0, T]\) and \(\omega, \alpha \in \Omega\).
    \item[(c)] 
    For every compact set \(\mathscr{C} \subset \Omega\) and every \(t \in [0, T]\), 
    \[
    \sup \Big\{ \| \mu^n (t, \omega) - \mu (t, \omega) \|_H + \|a^n (t, \omega) - a (t, \omega) \|_{L (U, H)} \colon \omega \in \mathscr{C} \, \Big\} \to 0
    \]
    as \(n \to \infty\).
\end{enumerate}
Furthermore, if Condition~\ref{cond: compact2} is in force, then \(b^n\) and \(a^n\) also satisfy the following:
\begin{enumerate}
\item[\textup{(d)}]  There is a constant \(C > 0\), that does not depend on \(n\), such that, for all \((t, \omega) \in [0, T] \times \Omega\), and \(k, n \in \mathbb{N}\), 
	\begin{align*}
	|\langle \mu^n (t, \omega),e_k\rangle_H |^2 
	+ \| (a^n)^* (t, \omega) e_k\|^2_{U}&\leq C \, \cCk^2\, \big[1+\|\omega\|_t^2\big].
\end{align*}
\end{enumerate}

Take a filtered probability space \(\mathbb{B} = (\Sigma, \mathcal{G}, (\mathcal{G}_t)_{t \in [0, T]}, P)\) that supports a standard cylindrical Brownian motion \(W\) and an \(\cG_0\)-measurable random variable \(\xi_0 \sim \nu\). 
For every \(n \in \mathbb{N}\), thanks to (a) and (b) above, we can use a standard contraction argument (cf., for example, \cite[Appendix~A]{C23}) to conclude the existence of a mild solution process (with continuous paths) on the driving system \((\mathbb{B}, W)\) to the SPDE 
\[
d Y^n_t = A Y^n_t \, dt + \mu^n (t, Y^n) \, dt + a^n (t, Y^n) \, d W_t, \quad Y^n_0 = \xi_0. 
\]
In the following, we will see that the laws of \((Y^n)_{n = 1}^\infty\) are tight and that any of its accumulation points is a solution measure (i.e., the law of a solution process) to the original SPDE under consideration.

\smallskip
{\em Step 2: Tightness}. By virtue of Condition~\ref{cond: compact1}, or the properties (d) that hold under Condition~\ref{cond: compact2}, it follows along the lines of the proof for Lemma~\ref{lem: Rk rel comp} that the laws of \((Y^n)_{n = 1}^\infty\) are tight (equivalently, relatively compact) in \(\W (\Omega)\). We omit the details for brevity.

\smallskip
{\em Step 3: The martingale problem argument}. By Step 2, up to passing to a subsequence, we can assume that \(Q^n := P \circ (Y^n)^{-1}\) converges weakly to a measure \(Q\). 
Take \(y^* \in D (A^*)\) and \(g \in C^2_c (\bR; \bR)\). To streamline our notation, we set \(\mu^0 := \mu\) and \(a^0 := a\). 
For \(n \in \mathbb{Z}_+\), we set 
\[
\mathsf{M}^n := g (\langle X, y^* \rangle_H ) - \int_0^\cdot \mathcal{L}^n g (s, X) \, ds, 
\]
where 
\begin{align*}
	\mathcal{L}^n g (s, X) := g (\langle X_s, y^* \rangle_H) \, \big( \langle  X_s&, A^* y^* \rangle_H + \langle \mu^n (s, X), y^* \rangle \big)
	\\&+ \tfrac{1}{2} g'' (\langle X_s, y^* \rangle_H) \| (a^n)^* (s, X) y^* \|^2_U.
\end{align*}
Here, recall that \(X\) denotes the coordinate process on \(\Omega\).

Our aim is to prove that \(\mathsf{M}^0\) is a local \(Q\)-martingale. In that case, a standard representation theorem for cylindrical local martingales (\cite[Theorem~3.1]{OndRep}), jointly with the relation of weak and mild solutions (\cite[Theorem~13]{MR2067962}), provides the existence of a martingale solution. 
In the remainder of this proof, we establish this local martingale property. 

\smallskip
For \(\ell > 0\), set 
\[
T_\ell := \inf \{t \in [0, T] \colon \|X_t\|_H \geq \ell \, \}. 
\]
By virtue of \cite[Lemma~11.1.2]{SV}, there exists a sequence \((\ell^n)_{n = 1}^\infty\) such that, on one hand, \(\ell^n \nearrow \infty\)  and, on the other hand, \(\omega \mapsto T_{\ell^n}(\omega)\) is \(Q\)-a.s. continuous for every \(n \in \mathbb{N}\). 

Take two times \(0 \leq a < b \leq T\) and a bounded continuous function \(\psi \colon \Omega \to \bR\) such that \(\psi (\omega)\) depends on \(\omega\) only through \((\omega (s))_{s \leq a}\). Let \(\ell > 0\) be such that \(\omega \mapsto T_\ell (\omega)\) is \(Q\)-a.s. continuous. 
In the following, we show that 
\begin{align} \label{eq: to prove MPA}
E^{Q} \big[ (\mathsf{M}^0_{b \wedge T_\ell} - \mathsf{M}^0_{a \wedge T_\ell} ) \, \psi \, \big] = 0.
\end{align}
It is clear that this implies the desired local \(Q\)-martingale property of \(\mathsf{M}^0\).

The continuous mapping theorem yields that 
\begin{align} \label{eq: step CMT}
	E^{Q} \big[ (\mathsf{M}^0_{b \wedge T_\ell} - \mathsf{M}^0_{a \wedge T_\ell} ) \, \psi \, \big]  = \lim_{n \to \infty} E^{Q^n} \big[ (\mathsf{M}^0_{b \wedge T_\ell} - \mathsf{M}^0_{a \wedge T_\ell} ) \, \psi \, \big].
\end{align}
Here, we use that \(\mathsf{M}^0_{\cdot \wedge T_\ell}\) is continuous by Condition~\ref{cond: existence1}~(i) and bounded by Condition~\ref{cond: existence1}~(ii) and the definition of \(T_\ell\).
Thanks to (c) and the dominated convergence theorem, it follows that \(\mathsf{M}^n_{t \wedge T_\ell} \to \mathsf{M}^0_{t \wedge T_\ell}\) uniformly on compact subsets of \(\Omega\), for every time \(t \in [0, T]\) that we fix in the following.
By the definition of the stopping time \(T_\ell\), Condition~\ref{cond: compact1} and part (a) above, there exists a constant \(\con  > 0\), independent of \(n\), such that 
\[
\big| \mathsf{M}^n_{t \wedge T_\ell} - \mathsf{M}^0_{t \wedge T_\ell} \big| \leq \con.
\]
Take \(\varepsilon > 0\). By the tightness of \((Q^n)_{n = 1}^\infty\), there exists a compact set \(\mathscr{K} \subset \Omega\) such that 
\[
\sup_{n \in \mathbb{N}} Q^n (\mathscr{K}^c) \leq \varepsilon.
\]
Therefore, we obtain that 
\begin{align*}
	E^{Q^n} \big[ \big| \mathsf{M}^0_{t \wedge T_\ell} - \mathsf{M}^n_{t \wedge T_\ell} \big|  \big] &\leq \con \, \varepsilon + \sup_{\omega \in \mathscr{K}} \big| \mathsf{M}^n_{t \wedge T_\ell (\omega)} (\omega) - \mathsf{M}^0_{t \wedge T_\ell (\omega)} (\omega) \big| \to \con \, \varepsilon
\end{align*}
as \(n \to \infty\). As \(\varepsilon > 0\) was arbitrary, we conclude that 
\[
\big| E^{Q^n} \big[ \mathsf{M}^0_{t \wedge T_\ell} \, \psi \, \big] - E^{Q^n} \big[ \mathsf{M}^n_{t \wedge T_\ell} \, \psi \, \big] \big| \leq  \|\psi\|_\infty \, E^{Q^n} \big[ \big| \mathsf{M}^0_{t \wedge T_\ell} - \mathsf{M}^n_{t \wedge T_\ell} \big|  \big] \to 0, 
\] 
and consequently, with \eqref{eq: step CMT}, 
\begin{align*}
	E^{Q} \big[ (\mathsf{M}^0_{b \wedge T_\ell} - \mathsf{M}^0_{a \wedge T_\ell} ) \, \psi \, \big]  = \lim_{n \to \infty} E^{Q^n} \big[ (\mathsf{M}^n_{b \wedge T_\ell} - \mathsf{M}^n_{a \wedge T_\ell} ) \, \psi \, \big].
\end{align*}
For every \(n \in \mathbb{N}\), \(\mathsf{M}^n_{\cdot \wedge T_\ell}\) is a \(Q^n\)-martingale by the construction of \(Q^n\). Hence, we get that
\[
E^{Q^n} \big[ (\mathsf{M}^n_{b \wedge T_\ell} - \mathsf{M}^n_{a \wedge T_\ell} ) \, \psi \, \big] = 0, \quad n \in \mathbb{N},
\]
which establishes \eqref{eq: to prove MPA}. The proof is complete.
\end{proof}

\end{document}